\documentclass[11pt]{article}
\usepackage[margin=1in]{geometry}
\usepackage{amsmath,amssymb,amsthm,booktabs,longtable,array,calc,tabularx}
\usepackage{graphicx}
\usepackage{microtype}
\usepackage{enumitem}
\usepackage{placeins}
\usepackage[hidelinks]{hyperref}

\usepackage[T1]{fontenc}
\usepackage[utf8]{inputenc}
\newtheorem{theorem}{Theorem}[section]
\newtheorem{proposition}[theorem]{Proposition}
\newtheorem{lemma}[theorem]{Lemma}
\newtheorem{corollary}[theorem]{Corollary}
\theoremstyle{definition}
\newtheorem{definition}[theorem]{Definition}
\theoremstyle{remark}

\providecommand{\tightlist}{\setlength{\itemsep}{0pt}\setlength{\parskip}{0pt}}
\begin{document}
\title{Exact Labelled Layer-Order Classification for Period-2 Mountain--Valley Assignments on $2\times 2k$ Maps}
\author{Michael Fofonov\\
Alferov University, Saint Petersburg, Russia\\
\texttt{spicmaff@gmail.com}}
\date{}
\maketitle

\begin{abstract}
We study exact labelled layer orders for prescribed period-2 mountain--valley assignments on ordinary orthogonal $2\times 2k$ maps. A state records a strict total order of the labelled faces; fold sequences and symmetry classes are not counted. For $k\ge2$, the standard local flat-foldability condition partitions the 64 formal six-bit words into 48 locally obstructed and 16 locally admissible words. The admissible words split into an eight-word pleat regime with singleton state sets and two four-word constant-row regimes. For canonical $MMMMVV$, upper/lower pivot compatibility is exactly $a=1$, $a=m$, or $a=b$, where $m=2k-1$; for canonical $MVMMMM$, it is exactly $B_R(a)=B_R(b)$, with $B_R(1)=0$ and $B_R(p)=\lfloor p/2\rfloor$ for $p\ge2$. Exact labelled symmetry transports these canonical fibres to every admissible period word, and each compatible ordered pivot pair determines one labelled state. The resulting four state-set regimes have cardinalities $0$, $1$, $6k-5$, and $4k-3$ as a downstream corollary. The boundary $k=1$ is treated separately: after projection to the four physical crease bits, the state set is empty or a singleton according to $\delta_0$. Bounded exact computation independently validates finite instances and the supplementary tables, but the all-$k$ results follow from the symbolic arguments.
\end{abstract}

\section{Introduction}
Once a rectangular map is placed flat, feasibility still depends on the vertical ordering of overlapping labelled faces. For prescribed mountain--valley assignments, local crease directions induce precedence relations, while global nonintersection imposes additional nesting constraints on the resulting order \cite{Nishat2013,NishatWhitesides2013}. This makes the layer-order set itself a natural object: existence asks whether the set is empty, whereas enumeration records only its cardinality.

We consider ordinary orthogonal $2\times 2k$ maps whose mountain--valley labels repeat with period two in the horizontal, upper-vertical, and lower-vertical crease families. The formal assignment is a six-bit word $P=(h_0,h_1,u_0,u_1,l_0,l_1)$, but these six bits are periodic labels rather than six physical creases. Our primary object is the exact set $F(P,k)$ of globally valid strict total orders of all labelled faces, with no quotient by symmetry and no counting of fold sequences.

Period two is restrictive enough that, in the constant-row branches, each physical row reduces to a one-parameter pivot normal form, while the interaction between the two row orders remains global and produces two distinct compatibility mechanisms. This balance is what makes an explicit state-set description possible without reducing the problem to a purely local count.

For $k\ge2$, the classification is organized by the sign invariants $\delta_i=h_0h_1u_il_i$, $\tau=u_0u_1$, and $\rho=h_0h_1$. A standard local condition eliminates 48 of the 64 formal words, and exact labelled symmetries reduce the remaining global analysis to pleat states and two canonical constant-row families. The $MMMMVV$ family is governed by a boundary-or-diagonal pivot criterion, whereas the $MVMMMM$ family is governed by equality of a two-site block coordinate. These mechanisms assemble into an exact four-regime state-set theorem, with the formulas $0$, $1$, $6k-5$, and $4k-3$ appearing only afterward as cardinality consequences.

We use established final-state validity and order-extension frameworks, together with the known uniform-sign strip family, as starting points for the analysis below. Proposition~\ref{obj:P2} records the row orders in explicit seam-indexed form and gives a self-contained derivation in the notation needed for the two-row problem; no priority claim is made for the underlying one-dimensional classification.

The family-specific results proved here begin with the interaction of the two row states. We give exact labelled transports among the period words, explicit compatibility criteria for the two constant-row canonical families, and a multiplicity-one correspondence between compatible ordered pivot pairs and labelled states after the remaining two-dimensional nonintersection conditions are verified. The all-word state-set classification is then assembled from those fibres, and the count formulas are downstream cardinality consequences. The width-two case $k=1$ is handled separately because two formal period bits have no physical crease carriers there.

The proof follows the same structural order as the statement: local reduction, exact symmetries, the pleat singleton mechanism, row-local pivot normal forms, the two canonical constant-row compatibility theorems, and final all-word synthesis. Technical interval and uniqueness arguments are retained in the appendices so that the main text can display the mechanisms without replacing proof obligations by intuition. Finite exact enumeration is used as an independent validation tool, not as a premise for any all-$k$ statement.

\section{Related work}
\subsection{Layer-order validity for rectangular maps}
Nishat gives a validity characterization for labelled rectangular-map orders that combines the directed mountain--valley precedence network with twin-butterfly stacking or nesting conditions, together with an exact-order validation procedure \cite{Nishat2013}. Nishat and Whitesides provide a shorter $2\times n$ publication-lineage treatment of the same map-folding setting \cite{NishatWhitesides2013}. We use these ingredients as background semantics, including the degree-four local flat-foldability condition whose period-two specialization yields the two $\delta_i$ equations \cite{Nishat2013,NishatWhitesides2013}.

\subsection{\texorpdfstring{Direct $2\times n$ and all-state frameworks}{Direct 2 x n and all-state frameworks}}
Morgan's 2012 thesis gives a polynomial-time algorithm for deciding flat foldability of general assigned $2\times n$ maps and, in its elementary cases, already supplies existence for the two locally admissible structural branches that arise after the period-two local reduction \cite[Section~2.2.2]{Morgan2012}. Thus nonemptiness of the locally admissible period-two assignments is not claimed here as new; the issue studied below is their exact labelled state sets. Jia and Mitani formulate fixed-MV $2\times n$ folding as an order-extension problem over a $1\times2n$ strip representation with map-specific nonintersection filtering \cite{JiaMitani2026}. Their fixed-assignment procedure is an existence/decision framework; their broader enumeration procedure enumerates MV patterns rather than the multiplicity of every labelled layer order of one fixed pattern. Jia, Mitani, and Uehara also study the related $2\times n$ box-pleated geometry without a prescribed mountain--valley assignment \cite{JiaMitaniUehara2020}. These works provide direct algorithmic context but do not state the period-two fibre criteria proved below.

Akitaya, Demaine, and Ku give a broader facewise layer-order constraint formulation, decompose the residual constraint system into independent components, and preserve multiplicity through Cartesian-product assembly of component solutions \cite{AkitayaDemaineKu2026}. Accordingly, general exact-state enumeration and multiplicity-preserving decomposition are treated here as established framework-level ideas rather than as contributions of the period-two specialization.

\subsection{Fixed-MV strips and pleat context}
One-dimensional fixed-MV stamp folding has an established counting and structural literature. Asano et al. prove that a unit-spaced strip has a unique folded state if and only if its MV pattern is a pleat \cite{AsanoEtAl2010}. For the uniform-sign family, Hull et al. prove $c(M^q)=c(V^q)=q$, describe the corresponding two-spiral states, and use the adjacent pair of extreme faces as the structural parameter \cite[Theorem~2.1 and the proof of Theorem~3.1]{HullEtAl2025}. With $q=2k-1$, this is the one-dimensional family underlying Proposition~\ref{obj:P2}; the formulas $R^\pm(k,p)$ below are an explicit seam-indexed reformulation chosen for the subsequent two-row analysis, not a new one-dimensional classification. Hoshido et al. provide later fixed-MV strip context \cite{HoshidoEtAl2025}. Uehara's 2011 chapter and Umesato et al. 2013 remain part of that historical literature \cite{Uehara2011,UmesatoEtAl2013}; we do not use their unavailable full bodies as sole proposition-level authority. We therefore make no priority claim for either the underlying uniform-sign row family or the pleat mechanism.

\subsection{Algebraic layer-order work}
The fixed-row horizontal reconstruction used in the $MVMMMM$ branch is an instance of the classical single-stack problem; in permutation language, the generic three-label obstruction is the standard $231$ obstruction to one-stack sortability \cite{Ulfarsson2012}. The work specific to the present family is the evaluation of that stack condition on the pivot-indexed row words, together with the proof that a surviving horizontal merge satisfies the remaining two-dimensional map constraints.

Ku, Terao, and Terao have a recent algebraic treatment of layer-order constraints in the same broad methodological lineage \cite{KTT2026}. The full primary chapter body was not accessible to us for proposition-level comparison, so we do not infer which of the explicit period-two statements may follow from that work and make no priority claim relative to it. The distinction maintained throughout this paper is therefore between established general validity, strip, stack, and solver frameworks and the explicit symbolic specialization proved here, not between possible and impossible prior approaches.

\section{Model, semantics, and period-2
notation}\label{model-semantics-and-period-2-notation}

Figure~\ref{fig:m4-f1} summarizes the physical face and crease notation used below; the definitions remain textual and do not rely on the schematic.

\begin{figure}[htbp]
\centering
\includegraphics[width=.92\linewidth]{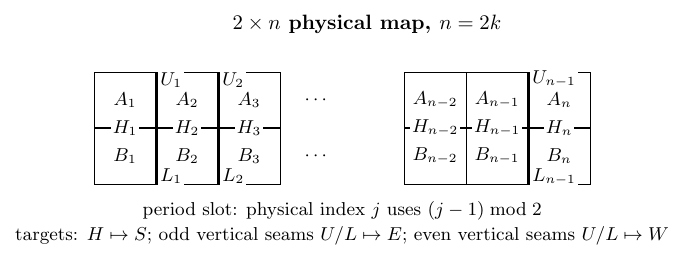}
\caption{Schematic of the physical $2\times n$ map ($n=2k$). Faces remain labelled. $H_c$ is the horizontal crease between $A_c$ and $B_c$; $U_i$ and $L_i$ are the physical shared seams between adjacent upper and lower faces. Physical index $j$ uses period slot $(j-1)\bmod 2$. The six formal period bits specify repeating labels, not six physical creases.}
\label{fig:m4-f1}
\end{figure}

\begin{definition}[Exact labelled layer-order states]\label{obj:D1}
Fix \(k\ge2\) and put \(n=2k\). The upper faces are \(A_1,\ldots,A_n\)
and the lower faces are \(B_1,\ldots,B_n\). The physical creases are

\begin{itemize}
\tightlist
\item
  \(H_c=(A_c,B_c)\) for \(1\le c\le n\);
\item
  \(U_i=(A_i,A_{i+1})\) for \(1\le i\le n-1\);
\item
  \(L_i=(B_i,B_{i+1})\) for \(1\le i\le n-1\).
\end{itemize}

A period-2 mountain--valley assignment is encoded by \[
P=(h_0,h_1,u_0,u_1,l_0,l_1)\in\{M,V\}^6.
\] The crease with physical index \(j\) uses period slot
\((j-1)\bmod2\): \(H_c\) receives \(h_{(c-1)\bmod2}\), and \(U_i,L_i\)
receive \(u_{(i-1)\bmod2},l_{(i-1)\bmod2}\).

The validity semantics below are the specialization of the
labelled-order framework for rectangular map folding \cite{Nishat2013,NishatWhitesides2013}. A physical flat-folded state of the prescribed map is represented by a strict labelled total order of all \(4k\) faces. Two logically distinct requirements apply. First, the prescribed crease assignment itself must satisfy ordinary local flat-foldability at every present interior vertex; this is a prerequisite on the assignment \(P\), independent of any proposed total order. Throughout, $x\prec y$ means that face $x$ lies above face $y$; displayed total orders are written from top to bottom. We consider final zero-thickness layer orders in this rectangular-map model, not the existence of a prescribed continuous folding motion. Second, a candidate total order must satisfy the following three \emph{global order-validity} conditions.

\begin{enumerate}
\def\labelenumi{\arabic{enumi}.}
\tightlist
\item
  \textbf{Directed M/V precedence.} Identify \(A_c\) with row \(0\) and
  \(B_c\) with row \(1\). A face \((r,c)\) is light-up exactly when
  \((r+c-1)\bmod2=0\). Across an \(M\) crease, the light-up wing
  precedes the dark-up wing; across a \(V\) crease, the dark-up wing
  precedes the light-up wing.
\item
  \textbf{Folded target classes.} Every \(H_c\) has target class \(S\).
  Each \(U_i\) and \(L_i\) has target class \(E\) for odd \(i\) and
  \(W\) for even \(i\).
  The symbols \(S,E,W\) identify target classes in the folded geometry: only creases assigned to the same target class can enter the same disjoint-wing Butterfly test in the next clause. Thus \(S\) collects horizontal creases, while \(E\) and \(W\) distinguish the two vertical-seam parities.
\item
  \textbf{Butterfly/noncrossing validity.} For two distinct creases with
  disjoint wing sets and the same target class, let their endpoint
  positions be \(a_0<a_1\) and \(b_0<b_1\). The pair is invalid exactly
  when the endpoints strictly alternate, \[
  a_0<b_0<a_1<b_1
  \qquad\text{or}\qquad
  b_0<a_0<b_1<a_1.
  \] Equivalently, the two endpoint intervals must be nested or
  disjoint. Creases sharing a face are not part of this disjoint-wing
  test.
\end{enumerate}

Let \(F(P,k)\) be the exact set of physical globally valid labelled layer-order states for the period-expanded map, represented by their strict total orders. The local flat-foldability requirement is not a fourth order predicate, a state coordinate, a pivot parameter, or an auxiliary choice: it constrains the prescribed assignment before any candidate order is tested. Consequently, an assignment that fails local flat-foldability has no physical state. For a locally admissible assignment, candidate-order membership is tested by the three global order-validity clauses above. Faces remain labelled and no symmetry quotient is taken. Fold sequences, auxiliary certificates, pivot parameters, and computational witnesses are not states.
\end{definition}
\begin{definition}[Period-2 invariants]\label{obj:D2}
Use the sign convention \(M=+1\), \(V=-1\) and multiplication of signs.
Define \[
\rho=h_0h_1,\qquad
\tau=u_0u_1,\qquad
\delta_i=h_0h_1u_i l_i=\rho u_i l_i\quad(i=0,1),
\] and put \[
m=n-1=2k-1.
\] Further coordinates are introduced only when they become useful: row
pivots in Section 6, the MVMMMM block coordinate \(B_R\) in Section 7.2,
and the reflected coordinate \(B_L\) in Section 8 after the physical
seam reflection has been established.
\end{definition}
\section{Local reduction and exact labelled
symmetries}\label{local-reduction-and-exact-labelled-symmetries}

Definition 3.1 separates assignment-level physical admissibility from candidate-order validity. The three displayed order conditions test a proposed labelled total order; the local degree-4 rule is instead a necessary physical condition on the prescribed crease assignment itself. The next lemma specializes that prerequisite to the period-two words. Because the local rule depends only on \(P\), not on a proposed order, failure at either physical parity makes \(F(P,k)\) empty before any global pivot or merge analysis begins; no new state coordinate or order predicate is introduced.

\begin{lemma}[Period-2 local obstruction]\label{obj:L1}
For every \(P\in\{M,V\}^6\) and every \(k\ge2\), \[
F(P,k)\ne\varnothing
\quad\Longrightarrow\quad
\delta_0=\delta_1=-1.
\] Equivalently, if \(\delta_0\ne-1\) or \(\delta_1\ne-1\), then
\(F(P,k)=\varnothing\).
\end{lemma}
\begin{proof}
 At the separator between columns \(c\) and \(c+1\), the
ordinary degree-\(4\) interior vertex is incident to
\(H_c,H_{c+1},U_c,L_c\). The standard local flat-foldability rule
\cite{Nishat2013,NishatWhitesides2013} requires an odd number
of mountain creases. Under the sign convention this is equivalent to \[
H_cH_{c+1}U_cL_c=-1.
\] With \(i=(c-1)\bmod2\), period expansion gives \[
h_i h_{1-i}u_i l_i=h_0h_1u_i l_i=\delta_i.
\] For \(k\ge2\), both separator parities occur physically; at the
minimum width \(k=2\), the separators \(1,2,3\) already have parities
\(0,1,0\). Thus both equations are necessary.
\end{proof}

\begin{corollary}[The $48/16$ local split]\label{obj:C1}
Exactly \(16\) of the \(64\) formal period words satisfy
\(\delta_0=\delta_1=-1\). The other \(48\) have empty state set for
every \(k\ge2\).
\end{corollary}
\begin{proof}
 Choose \(h_0,h_1,u_0,u_1\) arbitrarily. The two local
equations then determine \[
l_i=-\rho u_i\qquad(i=0,1)
\] uniquely. Hence \(2^4=16\) words survive the local test. The literal
\(64\)-word table is only a finite coverage check and is not a premise
of this argument or of the all-\(k\) classification.
\end{proof}

The remaining proofs repeatedly move between period words. Because
states are labelled, the required symmetry statements must be exact
state-set bijections rather than statements ``up to symmetry.''

\begin{proposition}[Exact labelled symmetries]\label{obj:P1}
Let \(n=2k\) be even. Define \[
\phi_R(A_c)=B_c,\qquad \phi_R(B_c)=A_c,
\] \[
\phi_H(A_c)=A_{n+1-c},\qquad \phi_H(B_c)=B_{n+1-c}.
\] For a total order \(\Pi\), let \(\phi(\Pi)\) denote elementwise
relabelling and \(\operatorname{rev}(\Pi)\) literal reversal. The
following transformations are exact labelled state-set bijections:

\begin{longtable}[]{@{}
  >{\raggedright\arraybackslash}p{(\columnwidth - 4\tabcolsep) * \real{0.3333}}
  >{\raggedright\arraybackslash}p{(\columnwidth - 4\tabcolsep) * \real{0.3333}}
  >{\raggedright\arraybackslash}p{(\columnwidth - 4\tabcolsep) * \real{0.3333}}@{}}
\toprule\noalign{}
\begin{minipage}[b]{\linewidth}\raggedright
\(g\)
\end{minipage} & \begin{minipage}[b]{\linewidth}\raggedright
word action \(gP\)
\end{minipage} & \begin{minipage}[b]{\linewidth}\raggedright
total-order action \(T_g\)
\end{minipage} \\
\midrule\noalign{}
\endhead
\bottomrule\noalign{}
\endlastfoot
\(C\) & \((-h_0,-h_1,-u_0,-u_1,-l_0,-l_1)\) &
\(\operatorname{rev}(\Pi)\) \\
\(R\) & \((h_0,h_1,l_0,l_1,u_0,u_1)\) &
\(\operatorname{rev}(\phi_R(\Pi))\) \\
\(H\) & \((h_1,h_0,u_0,u_1,l_0,l_1)\) &
\(\operatorname{rev}(\phi_H(\Pi))\) \\
\(CR\) & \((-h_0,-h_1,-l_0,-l_1,-u_0,-u_1)\) & \(\phi_R(\Pi)\) \\
\(CH\) & \((-h_1,-h_0,-u_0,-u_1,-l_0,-l_1)\) & \(\phi_H(\Pi)\) \\
\end{longtable}
\addtocounter{table}{-1}

For every listed \(g\), \[
\Pi\in F(P,k)
\quad\Longleftrightarrow\quad
T_g(\Pi)\in F(gP,k),
\] and therefore \[
F(gP,k)=T_g(F(P,k))
\] as literal labelled sets.

Under horizontal reflection, the physical seam between columns \(p\) and
\(p+1\) is sent to the seam between columns \(n-p\) and \(n-p+1\). Thus
the seam coordinate transforms as \[
p\longmapsto n-p.
\]
\end{proposition}
\begin{proof}
 Each face map preserves crease incidence, disjointness,
and folded target class. Complement \(C\) reverses every directed M/V
precedence by complementing the crease labels. Row swap \(R\) and
horizontal reflection \(H\) reverse directed precedences by flipping
checkerboard orientation. Hence \(C,R,H\) require literal order
reversal, while the two reversal mechanisms cancel in \(CR\) and \(CH\).
Strict endpoint alternation is preserved by bijective relabelling and by
order reversal. Each transformation is involutive, giving both
directions of the state-set equivalence. Appendix B records the literal
crease and parity maps.
\end{proof}

\section{Pleat singleton regime}\label{pleat-singleton-regime}

When \(\tau=-1\), the two row directions alternate in the way that turns
primitive precedences into a full chain. This makes uniqueness
transparent, but uniqueness is not yet existence: the forced chain still
has to satisfy every long-range Butterfly condition.

\begin{theorem}[Pleat singleton states]\label{obj:T1}
Assume \(k\ge2\), \(\delta_0=\delta_1=-1\), and \(\tau=-1\). For \[
P_+=MMMVVM,
\] put \[
Q_q=(A_{2q-1},B_{2q-1},B_{2q},A_{2q}),\qquad 1\le q\le k,
\] and \[
\Omega_+(k)=Q_1Q_2\cdots Q_k.
\] For \[
P_-=MVMVMV,
\] put \[
\Omega_-(k)=(A_1,A_2,\ldots,A_n,B_n,B_{n-1},\ldots,B_1).
\] Then \[
F(MMMVVM,k)=\{\Omega_+(k)\},\qquad
F(MVMVMV,k)=\{\Omega_-(k)\}.
\] The other six locally admissible pleat words have the exact singleton
state sets \[
\begin{array}{c|c}
P & F(P,k)\\\hline
MMVMMV & \{T_R(\Omega_+(k))\}\\
VVMVVM & \{T_{CR}(\Omega_+(k))\}\\
VVVMMV & \{T_C(\Omega_+(k))\}\\
VMMVMV & \{T_H(\Omega_-(k))\}\\
VMVMVM & \{T_C(\Omega_-(k))\}\\
MVVMVM & \{T_{CH}(\Omega_-(k))\}.
\end{array}
\]
\end{theorem}
\begin{proof}
 For \(P_+\), the primitive precedences contain \[
A_{2q-1}\prec B_{2q-1}\prec B_{2q}\prec A_{2q}
\] inside each block and \(A_{2q}\prec A_{2q+1}\) between consecutive
blocks. These edges form a directed Hamiltonian path through all \(4k\)
faces, so every precedence-compatible total order equals
\(\Omega_+(k)\).

For \(P_-\), the relations \[
A_c\prec A_{c+1}\ (1\le c<n),\qquad
A_n\prec B_n,\qquad
B_{c+1}\prec B_c\ (1\le c<n)
\] form the Hamiltonian path \(A_1,\ldots,A_n,B_n,\ldots,B_1\), and
hence force \(\Omega_-(k)\).

These Hamiltonian chains establish uniqueness only. Appendix C verifies
full validity for arbitrary separation in all same-target families
\(H/H\), \(U/U\), \(L/L\), and \(U/L\). Thus both forced orders exist as
globally valid states. The exact symmetry proposition then transports
the two singleton states to the remaining six pleat words.
\end{proof}

\section{Row-local pivot normal
forms}\label{row-local-pivot-normal-forms}

The case \(\tau=+1\) does not collapse to a single chain. Instead, each
physical row has a rigid one-parameter family of admissible orders. The
parameter is a \textbf{pivot}: the vertical seam joining the first and
last faces of the row order. Once the two row pivots are fixed, the
global problem becomes a compatibility-and-merge problem between two
prescribed row chains.

The underlying uniform-sign one-dimensional family is known in the stamp-folding literature: for a strip with $q$ uniformly assigned creases, Hull et al. obtain exactly $q$ labelled states and a two-spiral description indexed by the adjacent extreme faces \cite[Theorem~2.1 and the proof of Theorem~3.1]{HullEtAl2025}. Proposition~\ref{obj:P2} records the same row family in the explicit seam coordinate and orientation conventions needed below. Its purpose here is to provide a self-contained row-normal-form interface for the two-dimensional compatibility problem, not to assert priority for the one-dimensional classification.

For one row of width \(2k\), write its faces as \(F_1,\ldots,F_{2k}\)
and set \[
O_i=F_{2i-1},\qquad D_i=F_{2i}\qquad(1\le i\le k).
\] The odd seam \(2i-1\) joins \((O_i,D_i)\) and has target class \(E\);
the even seam \(2i\) joins \((D_i,O_{i+1})\) and has target class \(W\).

For effective sign \(+\), impose the row-local directed constraints \[
O_i\prec D_i\quad(1\le i\le k),\qquad
O_{i+1}\prec D_i\quad(1\le i<k),
\] together with noncrossing of all same-target seam spans within the
row. For sign \(-\), reverse the directed inequalities. Let
\(R_{\mathrm{loc}}^\sigma(k)\) denote the resulting set of row-local
orders.

For \(1\le p\le2k-1\), define \(R^+(k,p)\) by \[
R^+(k,2s-1)=
(O_s,O_{s-1},\ldots,O_1,
D_1,\ldots,D_{s-1},
O_{s+1},\ldots,O_k,
D_k,\ldots,D_{s+1},D_s),
\] and \[
R^+(k,2s)=
(O_{s+1},\ldots,O_k,
D_k,\ldots,D_{s+1},
O_s,\ldots,O_1,
D_1,\ldots,D_s),
\] with empty blocks omitted. Set \[
R^-(k,p)=\operatorname{rev}(R^+(k,p)).
\]

\begin{proposition}[Row-local pivot normal forms]\label{obj:P2}
For every \(k\ge2\) and \(\sigma\in\{+,-\}\), \[
R_{\mathrm{loc}}^\sigma(k)
=
\{R^\sigma(k,p):1\le p\le2k-1\},
\] and the pivot \(p\) is unique. It is the physical seam whose incident
faces are the first and last faces of the row order.

If \(R_{\mathrm{glob}}^{P,r}(k)\) denotes the row restrictions induced
by globally valid states of a fixed word \(P\), then \[
R_{\mathrm{glob}}^{P,r}(k)\subseteq R_{\mathrm{loc}}^\sigma(k).
\] No converse is asserted. A global state must restrict to one of these
row-local normal forms, but a row-local pivot need not extend to a
global state.
\end{proposition}
\begin{proof}
 Appendix A proves the adjacent-span trichotomy, the
normal-form relation word, uniqueness of the extreme-face seam, and the
local/global scope boundary.
\end{proof}

For illustration only, take \(k=3\) and the odd pivot \(p=3\) (so \(s=2\)). The formula gives
\[
R^+(3,3)=(O_2,O_1,D_1,O_3,D_3,D_2)=(F_3,F_1,F_2,F_5,F_6,F_4).
\]
The first and last faces are therefore \(F_3\) and \(F_4\), exactly the two faces incident to seam \(3\). This example illustrates the pivot coordinate and is not a global-realizability argument: Proposition 6.1 still asserts only the inclusion \(R_{\mathrm{glob}}\subseteq R_{\mathrm{loc}}\).

For later use, let \(W_p\) denote the column sequence of \(R^+(k,p)\).
Thus \(W_p\) records the row order with the row label suppressed;
\(A(W_p)\) and \(B(W_p)\) mean the corresponding upper and lower
labelled row orders.

\textbf{Pivot coordinates under the exact symmetries.} Now that pivots
have been defined, Proposition 4.3 gives their literal transport.
Complement \(C\) leaves the numerical pivot unchanged. Row exchange
\(R\) and \(CR\) swap the ordered upper/lower pivots, \[
(a,b)\longmapsto(b,a).
\] Horizontal reflection \(H\) and \(CH\) apply the physical seam
reflection to each row, \[
(a,b)\longmapsto(n-a,n-b).
\] These are coordinate consequences of the exact labelled state
bijections; they do not assert global realizability of arbitrary
row-local pivots.

Table~\ref{tab:m4-tbl3} collects the exact symmetry actions, including the pivot consequences now that pivots have been defined; it summarizes Proposition~\ref{obj:P1} and does not replace its bijection proof.

\begin{table}[htbp]
\centering
\caption{Exact symmetry actions, including pivot-coordinate consequences. This table is placed only after pivots have been defined.}
\label{tab:m4-tbl3}
\resizebox{\linewidth}{!}{\begin{tabular}{llll}
\toprule
Action & Word action & Face/order action & Pivot action \\
\midrule
C & global M/V complement & face labels fixed; reverse total order & (a,b)->(a,b) \\
R & swap U/L period pairs & swap A\_c<->B\_c; reverse total order & (a,b)->(b,a) \\
H & swap h0,h1 & c->n+1-c in each row; reverse total order & (a,b)->(n-a,n-b) \\
CR & complement after row swap & swap A\_c<->B\_c; no order reversal & (a,b)->(b,a) \\
CH & complement after horizontal reflection & c->n+1-c; no order reversal & (a,b)->(n-a,n-b) \\
\bottomrule
\end{tabular}
}
\end{table}

\section{Canonical constant-row
families}\label{canonical-constant-row-families}

The four six-letter canonical words used repeatedly below are collected in Table~\ref{tab:m4-tbl2}; the table is a reading key for the pleat and constant-row mechanisms, not an additional classification claim.

\begin{table}[htbp]
\centering
\caption{Canonical representatives and their structural mechanisms.}
\label{tab:m4-tbl2}
\resizebox{\linewidth}{!}{\begin{tabular}{llllll}
\toprule
Word & $\rho$ & $\tau$ & Mechanism & Compatibility & Multiplicity \\
\midrule
MMMVVM & +1 & -1 & pleat Hamiltonian precedence + full Butterfly audit & N/A & one exact labelled state \\
MVMVMV & -1 & -1 & pleat row-monotone/reverse-row chain + full Butterfly audit & N/A & one exact labelled state \\
MMMMVV & +1 & +1 & row-local pivots + global merge & a=1 OR a=m OR a=b & one labelled state per compatible ordered pair \\
MVMMMM & -1 & +1 & LIFO horizontal reconstruction + pivot blocks & B\_R(a)=B\_R(b) & one labelled state per compatible ordered pair \\
\bottomrule
\end{tabular}
}
\end{table}

The local reduction and symmetries leave two essentially different
constant-row mechanisms. In both, the pivot records the only possible
row-local order. What remains is to decide which ordered pivot pairs can
be merged into a global state and whether such a merge is unique.

\subsection{\texorpdfstring{The canonical word
\(MMMMVV\)}{The canonical word MMMMVV}}\label{the-canonical-word-mmmmvv}

For \(P=MMMMVV\), both physical rows use the \(+\) normal form. A pivot
pair \((a,b)\) therefore prescribes the two row chains \(A(W_a)\) and
\(B(W_b)\). Let \(\operatorname{FullCompat}(k;a,b)\) mean that at least
one globally valid labelled state has exactly these row restrictions.

The criterion has a simple shape: every diagonal pair works, and every
pair with the upper pivot at either boundary works. The proof separates
obstruction, construction, full validity, and uniqueness; in particular,
local legality of the next horizontal event is not the same as existence
of a completable merge.

\begin{theorem}[Canonical $MMMMVV$ compatibility and multiplicity one]\label{obj:T2}
Let \(k\ge2\), \(m=2k-1\), and \(1\le a,b\le m\). Then \[
\operatorname{FullCompat}(k;a,b)
\quad\Longleftrightarrow\quad
(a=1)\ \text{or}\ (a=m)\ \text{or}\ (a=b).
\] Every compatible ordered pivot pair determines exactly one globally
valid labelled state. Conversely, every globally valid labelled
\(MMMMVV\) state determines a unique ordered pivot pair.
\end{theorem}
\begin{proof}

\textbf{1. Interior off-diagonal obstruction.} Write \(o_r=2r-1\) and
\(e_r=2r\). For a row word \(W_p\), define \[
X_p(i,j)\iff e_j\text{ precedes }o_i\text{ in }W_p.
\] The normal-form formulas give \[
X_{2s-1}(i,j)\iff i>s\text{ and }j<s,
\] \[
X_{2s}(i,j)\iff i\le s\text{ and }j>s.
\] If \(1<a<m\) and \(a\ne b\), Appendix~\ref{d.2.-necessity-the-four-edge-cycle} supplies a parity-complete
witness with \(X_a(i,j)=1\) and \(X_b(i,j)=0\). With \(o=2i-1\) and
\(e=2j\), the two row orders and the horizontal \(MMMMVV\) precedences
force the four-cycle \[
A_e\prec A_o\prec B_o\prec B_e\prec A_e,
\] so no global state exists.

\textbf{2. Allowed constructions.} Appendix~\ref{d.3.-five-explicit-existence-families} gives five explicit
families: the diagonal construction and four boundary constructions.
Their pivot pairs cover exactly \(a=b\), \(a=1\), and \(a=m\), including
all endpoint cases at \(k=2\).

\textbf{3. Full validity.} Each construction has the required row
restrictions and a laminar horizontal stack; Appendix~\ref{d.4.-horizontal-stack-validity-of-the-constructions} records the horizontal-stack check. Row-local validity supplies all within-row vertical conditions. Appendix~\ref{d.5.-remaining-ul-full-validity-obligations} discharges the remaining equal-parity $U_i/L_j$ Butterfly family at arbitrary separation. It first relates lower-face insertion slots to actual endpoint alternation, then uses the constant endpoint-sum structure of the four boundary shell families; the diagonal construction is handled separately through its contiguous column blocks and row-local noncrossing.

\textbf{4. Fixed-pivot uniqueness.} A partial horizontal merge can admit more than one H-locally legal next event even when only one choice has a complete H-valid continuation, so the proof never identifies local legality with completability. For diagonal pairs $a=b$, Lemma~\ref{lem:per2-diagonal} shows directly that the matching closer is forced after each opener and hence that $D(k,a)$ is the unique complete H-valid merge. For boundary pairs $a=1$ or $a=m$, Appendix~\ref{d.7.-projected-determinism-and-first-disagreement} compares any hypothetical second complete merge with the explicit construction from D.3--D.4 and takes their first divergence. Projected determinism leaves only an upper-odd versus lower-even opener choice. Appendix~\ref{d.8.-corrected-cross-parity-windows} lists exactly such windows along the known complete constructions, and Appendix~\ref{d.9.-wrong-switch-barriers} excludes the alternative branch by directed requirement cycles. The H-local/H-completability distinction is isolated in Appendix~\ref{d.6.-h-local-legality-is-not-h-completability}. Boundary and diagonal pairs exhaust the allowed set, so every allowed pivot pair has at most one full state, while the explicit constructions supply one.

\textbf{5. State--pivot bijection.} Every global state restricts to two
row-local normal forms, so Proposition 6.1 assigns it a unique ordered
pivot pair. The obstruction restricts the image to the compatible pairs;
existence and fixed-pair uniqueness give the inverse. Hence globally
valid states and compatible ordered pivot pairs are in bijection; Appendix~\ref{d.10.-global-state-to-pivot-bridge-and-multiplicity} records this bridge in the technical closure.
\end{proof}

\begin{figure}[htbp]
\centering
\includegraphics[width=.78\linewidth]{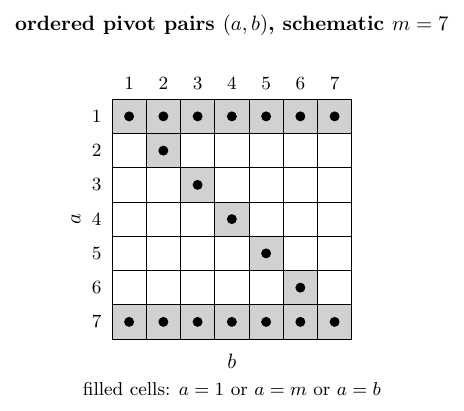}
\caption{Ordered-pair compatibility matrix for canonical $MMMMVV$, shown schematically at $m=7$. The horizontal coordinate is the lower-row pivot $b$ and the vertical coordinate is the upper-row pivot $a$. A filled cell is compatible exactly when $a=1$, $a=m$, or $a=b$; the two boundary conditions are therefore explicitly asymmetric in $a$.}
\label{fig:m4-f4}
\end{figure}

\subsection{\texorpdfstring{The canonical word
\(MVMMMM\)}{The canonical word MVMMMM}}\label{the-canonical-word-mvmmmm}

For \(P=MVMMMM\), the upper row uses \(R^+(k,a)\) and the lower row uses
\(R^-(k,b)\). Here the opposite row orientations make the horizontal
system especially rigid: upper faces open horizontal intervals, lower
faces close them, and \(H/H\) noncrossing becomes a LIFO reconstruction
problem.

Call a merge \textbf{H-compatible} if it satisfies all horizontal
precedences and all \(H/H\) Butterfly constraints. Let
\(\operatorname{HCompat}(k;a,b)\) and
\(\operatorname{FullCompat}(k;a,b)\) denote existence of an H-compatible
and a globally valid merge, respectively.

Define the right block coordinate \[
B_R(1)=0,\qquad
B_R(p)=\left\lfloor\frac p2\right\rfloor\quad(p\ge2).
\] Thus \[
Q_0=\{1\},\qquad Q_r=\{2r,2r+1\}\quad(1\le r\le k-1).
\] The coordinate \(B_R\) simply records which of these pivot blocks
contains \(p\).

\begin{theorem}[Canonical $MVMMMM$ block compatibility and multiplicity one]\label{obj:T3}
For \(k\ge2\) and \(1\le a,b\le m=2k-1\), \[
\operatorname{FullCompat}(k;a,b)
\quad\Longleftrightarrow\quad
B_R(a)=B_R(b).
\] For every compatible ordered pivot pair there is exactly one globally
valid labelled state, and every globally valid state determines one such
pair through its unique row-normal-form pivots.
\end{theorem}
\begin{proof}
 The proof has two independent steps.

\textbf{Step A: horizontal reconstruction is enough.} In \(MVMMMM\),
every horizontal precedence is upper-before-lower. Therefore upper faces
push their horizontal labels and lower faces pop them. Pairwise \(H/H\)
noncrossing is exactly the LIFO stack condition. With the two row chains
fixed, the next requested closer is fixed: if its opener has not yet
appeared, the upper chain is forced until it appears; if the label is on
top of the stack, its closer is forced; if it is buried, completion is
impossible. Hence an H-compatible merge, when it exists, is unique.

Successful H reconstruction is already globally valid. Proposition 6.1
supplies all within-row vertical precedences and same-row vertical
noncrossing. The only remaining same-target class is \(U_i/L_j\) with
\(i\equiv j\pmod2\). Appendix E proves the side-crossing transfer: a
crossing of such a cross-row pair would force a crossing in the
corresponding upper or lower same-row pair. Those are forbidden by the
row-normal forms. Thus \[
\operatorname{FullCompat}(k;a,b)
\quad\Longleftrightarrow\quad
\operatorname{HCompat}(k;a,b).
\]

\textbf{Step B: H compatibility is a block condition.} For
\(1\le r\le k-1\), write the two pivots in block \(r\) as \[
p_r^{(0)}=2r,\qquad p_r^{(1)}=2r+1.
\] Let \(c_r=2r+1\), \(d_r=2r+2\), and define \[
L_r=(2r-1,2r-3,\ldots,1,2,4,\ldots,2r),
\] \[
R_r=(2r+3,2r+5,\ldots,2k-1,2k,2k-2,\ldots,2r+4).
\] The row words have the endpoint normal forms \[
W_{p_r^{(0)}}=c_rR_rd_rL_r,
\qquad
W_{p_r^{(1)}}=c_rL_rR_rd_r.
\] Appendix E proves that unequal blocks always contain a three-label
LIFO obstruction, whereas each of the four ordered pivot pairs inside
one nonzero block has an explicit legal push/pop schedule; the singleton
block \(Q_0\) uses push-all/pop-all. Consequently \[
\operatorname{HCompat}(k;a,b)
\quad\Longleftrightarrow\quad
B_R(a)=B_R(b).
\] Combining Steps A and B gives the criterion. Step A gives fixed-pair
uniqueness, and Proposition 6.1 gives unique pivots for every global
state, so states and compatible ordered pivot pairs are in bijection.
\end{proof}

\begin{figure}[htbp]
\centering
\includegraphics[width=.55\linewidth]{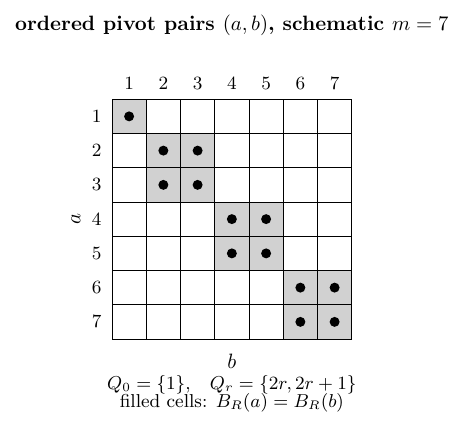}
\caption{Ordered-pair block compatibility matrix for canonical $MVMMMM$, shown schematically at $m=7$. The axes are the ordered pivots $(a,b)$. The singleton block is $Q_0=\{1\}$ and the nonzero blocks are $Q_r=\{2r,2r+1\}$. A filled cell satisfies $B_R(a)=B_R(b)$; each nonzero $2\times2$ block visibly contains four distinct ordered compatible pairs.}
\label{fig:m4-f5}
\end{figure}

\section{Exact transports and the all-64
classification}\label{exact-transports-and-the-all-64-classification}

The two canonical constant-row families are now completely understood.
Exact labelled symmetries transfer them to the remaining constant-row
words, after which the all-64 theorem is only a synthesis of the local
split and the three nonzero structural mechanisms.

Horizontal reflection has already been shown to send a physical seam by
\(p\mapsto n-p\). Only now define the reflected block coordinate \[
B_L(p):=B_R(n-p),\qquad 1\le p\le m.
\]

\begin{proposition}[Exact transports of the constant-row fibres]\label{obj:P3}
Assume \(k\ge2\), \(\delta_0=\delta_1=-1\), and \(\tau=+1\).

If \(\rho=+1\), the four words are \[
MMMMVV,\quad MMVVMM,\quad VVMMVV,\quad VVVVMM,
\] and their exact state sets are respectively
\[
\begin{aligned}
&F(MMMMVV,k),\quad T_R(F(MMMMVV,k)),\\
&T_{CR}(F(MMMMVV,k)),\quad T_C(F(MMMMVV,k)).
\end{aligned}
\] Writing \((\alpha,\beta)\) for the target upper/lower pivots, the
target compatibility criterion is \[
\alpha=1\ \text{or}\ \alpha=m\ \text{or}\ \alpha=\beta
\] for the canonical and complement targets, and \[
\beta=1\ \text{or}\ \beta=m\ \text{or}\ \alpha=\beta
\] for the row-swapped targets, because \(R\) and \(CR\) exchange the
ordered pivots.

If \(\rho=-1\), the four words are \[
MVMMMM,\quad VMMMMM,\quad VMVVVV,\quad MVVVVV,
\] and their exact state sets are respectively
\[
\begin{aligned}
&F(MVMMMM,k),\quad T_H(F(MVMMMM,k)),\\
&T_C(F(MVMMMM,k)),\quad T_{CH}(F(MVMMMM,k)).
\end{aligned}
\] For the canonical and complement targets, compatibility is
\(B_R(\alpha)=B_R(\beta)\). For the horizontally reflected targets, the
source pivots first transform geometrically as \[
(\alpha,\beta)=(n-a,n-b),
\] and only then the target criterion is expressed as \[
B_L(\alpha)=B_L(\beta).
\] Every compatible target pivot fibre is a singleton.
\end{proposition}
\begin{proof}
 Proposition 4.3 gives literal state-set bijections, and
the pivot consequences recorded after Proposition 6.1 give the exact
coordinate maps. Apply these bijections to the two canonical theorems.
Since both the state map and pivot map are bijections, each target pivot
fibre is the image of exactly one canonical fibre, so multiplicity one
transports fibrewise. No count equality is used.
\end{proof}

\begin{figure}[htbp]
\centering
\includegraphics[width=.80\linewidth]{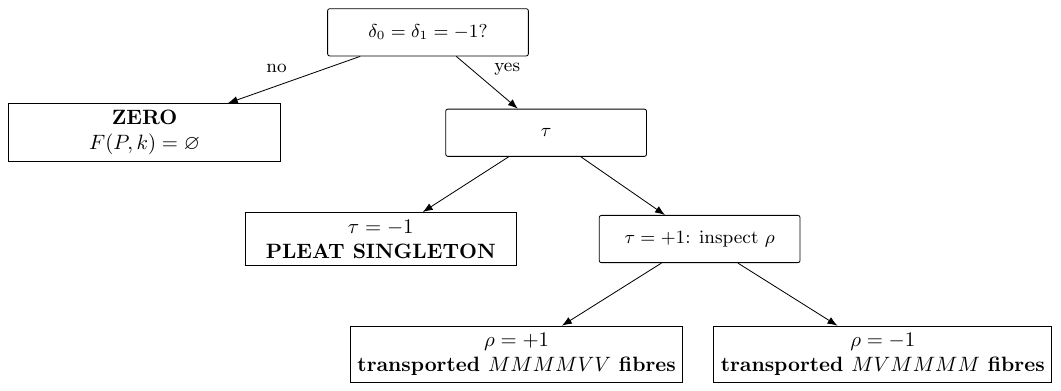}
\caption{Decision diagram for the exact state-set classification. Local admissibility is tested first by $\delta_0$ and $\delta_1$; $\tau$ separates the pleat and constant-row regimes, and $\rho$ separates the two constant-row canonical families. Counts are deliberately omitted because exact labelled state sets are primary.}
\label{fig:m4-f2}
\end{figure}
\FloatBarrier

\begin{theorem}[Exact labelled classification of all period-2 words]\label{obj:T4}
Let \(P\in\{M,V\}^6\) and \(k\ge2\). The exact labelled state set
\(F(P,k)\) is determined by exactly one of the following four regimes.

\begin{center}
\small
\begin{tabularx}{\linewidth}{@{}>{\raggedright\arraybackslash}p{0.22\linewidth}>{\raggedright\arraybackslash}p{0.29\linewidth}>{\raggedright\arraybackslash}X@{}}
\toprule
regime & condition & exact state set \\
\midrule
\textbf{Zero} & $\delta_0\ne-1$ or $\delta_1\ne-1$ & $F(P,k)=\varnothing$ \\
\textbf{Pleat singleton} & $\delta_0=\delta_1=-1$, $\tau=-1$ & the exact singleton from Theorem 5.1, transported by the labelled symmetries \\
\textbf{Constant row, $\rho=+1$} & $\delta_0=\delta_1=-1$, $\tau=+1$, $\rho=+1$ & the exact transported $MMMMVV$ fibres of Proposition 8.1, using the canonical criterion of Theorem 7.1; each compatible fibre is a singleton \\
\textbf{Constant row, $\rho=-1$} & $\delta_0=\delta_1=-1$, $\tau=+1$, $\rho=-1$ & the exact transported $MVMMMM$ fibres of Proposition 8.1, using the canonical criterion of Theorem 7.2; each compatible fibre is a singleton \\
\bottomrule
\end{tabularx}
\end{center}

In particular, the theorem classifies the literal labelled state sets,
not only their cardinalities.
\end{theorem}
\begin{proof}
 Lemma 4.1 excludes the \(48\) locally obstructed words,
and Corollary 4.2 leaves exactly \(16\) locally admissible words. Among
those, \(\tau=-1\) gives exactly the eight pleat words handled in
Section 5. If \(\tau=+1\), the sign \(\rho\) separates the remaining
eight words into the two groups of four transported from \(MMMMVV\) and
\(MVMMMM\). Proposition 8.1 gives the exact state-set transport in each
group. The alternatives are mutually exclusive and exhaustive, so all
\(64\) formal words are covered. The literal \(64\)-word table is a
finite cross-check only.
\end{proof}

\begin{table}[htbp]
\centering
\caption{Invariant/regime summary. The exact state-set description precedes the downstream count.}
\label{tab:m4-tbl1}
\resizebox{\linewidth}{!}{\begin{tabular}{llllll}
\toprule
Local condition & $\tau$ & $\rho$ & Type & Exact state set & Count \\
\midrule
either delta\_0 or delta\_1 != -1 & - & - & Zero & empty & 0 \\
delta\_0=delta\_1=-1 & -1 & either & Pleat singleton & exact singleton transported from one of two canonical pleat orders & 1 \\
delta\_0=delta\_1=-1 & +1 & +1 & Constant row: MMMMVV & exact transported singleton fibres indexed by compatible ordered pivots & 6k-5 \\
delta\_0=delta\_1=-1 & +1 & -1 & Constant row: MVMMMM & exact transported singleton fibres indexed by equal block coordinate & 4k-3 \\
\bottomrule
\end{tabular}
}
\end{table}

\begin{corollary}[Closed counts]\label{obj:C2}
Define, downstream of the exact state-set theorem, \[
C_P(k):=|F(P,k)|.
\] For every \(k\ge2\), \[
C_P(k)=
\begin{cases}
0, & \delta_0\ne-1\text{ or }\delta_1\ne-1,\\
1, & \delta_0=\delta_1=-1,\ \tau=-1,\\
6k-5, & \delta_0=\delta_1=-1,\ \tau=+1,\ \rho=+1,\\
4k-3, & \delta_0=\delta_1=-1,\ \tau=+1,\ \rho=-1.
\end{cases}
\]
\end{corollary}
\begin{proof}
 The zero and pleat values are immediate from the exact
state sets. For \(MMMMVV\), compatible pivot pairs are \[
\{(1,b):1\le b\le m\}\cup
\{(m,b):1\le b\le m\}\cup
\{(a,a):1\le a\le m\},
\] with only \((1,1)\) and \((m,m)\) counted twice, so there are
\(3m-2=6k-5\) singleton fibres. For \(MVMMMM\), \(Q_0=\{1\}\)
contributes one ordered pair and each of the \(k-1\) two-element blocks
contributes four, giving \(1+4(k-1)=4k-3\). Exact transports preserve
these cardinalities. No count enters the proof of Theorem 8.2.
\end{proof}

\section{\texorpdfstring{Boundary case
\(k=1\)}{Boundary case k=1}}\label{boundary-case-k1}

The width-\(2\) map is a separate physical branch. Only the creases
\(H_1,H_2,U_1,L_1\) exist; the formal bits \(u_1,l_1\) have no physical
carriers. Accordingly the generic two-parity argument, the row-pivot
machinery, and the \(k\ge2\) classification are not used here.

\begin{figure}[htbp]
\centering
\includegraphics[width=.92\linewidth]{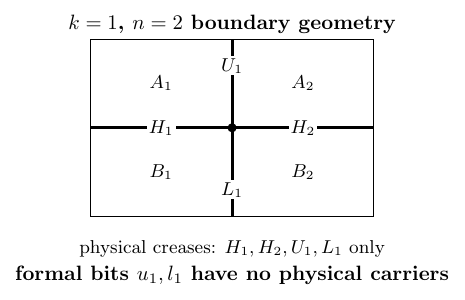}
\caption{Width-two boundary geometry for $k=1$. The physical map has four faces and exactly four creases $H_1,H_2,U_1,L_1$ meeting at the central degree-4 vertex. The formal period bits $u_1$ and $l_1$ have no physical carriers and therefore do not enter the $k=1$ criterion.}
\label{fig:m4-f7}
\end{figure}

\begin{proposition}[Exact boundary state set at $k=1$]\label{obj:P4}
Let \[
P=(h_0,h_1,u_0,u_1,l_0,l_1)\in\{M,V\}^6
\] and define the physical projection \[
\pi_1(P)=(h_0,h_1,u_0,l_0).
\] Formal words with the same \(\pi_1(P)\) define the same physical map
and the same labelled state set. Put \[
\delta_0(P)=h_0h_1u_0l_0.
\] Then \[
\delta_0=+1\Longrightarrow F(P,1)=\varnothing,
\] whereas \[
\delta_0=-1\Longrightarrow F(P,1)=\{\Omega(\pi_1(P))\}
\] for one exact labelled total order \(\Omega(\pi_1(P))\).
\end{proposition}
\begin{proof}
 The four physical creases form the incidence cycle \[
A_1-H_1-B_1-L_1-B_2-H_2-A_2-U_1-A_1.
\] Orient its edges by primitive M/V precedence. Relative to this cyclic
traversal, the product of the four orientation signs is \(\delta_0\).
When \(\delta_0=-1\), exactly one or three edges oppose the traversal;
the three majority-oriented consecutive edges form a directed
Hamiltonian path through all four faces, and the remaining edge is a
source-to-sink shortcut. Hence there is one linear extension. Its two
opposite-edge Butterfly pairs are respectively disjoint and nested, so
that order is globally valid.

When \(\delta_0=+1\), local flat-foldability fails. Appendix F also
gives the direct global-semantic obstruction: \(0\) or \(4\) traversal
reversals produce a directed cycle, while every \(2\)-reversal
orientation forces strict alternation in one opposite-edge pair.
\end{proof}

Consequently \[
|F(P,1)|=1\iff\delta_0=-1,
\qquad
|F(P,1)|=0\iff\delta_0=+1.
\] Appendix F lists the eight admissible physical projections and their
exact singleton orders.

\section{Finite validation and reproducibility}\label{finite-validation-and-reproducibility-firewall}

The finite computations provide independent validation checks; the all-$k$ classification is established by the symbolic arguments in the preceding sections. The independent finite validation compares exact labelled state sets, rather than counts alone, so that it can detect errors in labels, transports, pivot recovery, or noncrossing semantics that happen to preserve cardinality.

At $k=2$ the finite verification covers all 64 formal period words, the exact pivot fibres of the canonical $MMMMVV$ and $MVMMMM$ families, and all five exact symmetry actions. A separate exhaustive check covers all 64 formal words at $k=1$ and verifies invariance under the physically absent $u_1,l_1$ bits. The literal all-word tables are supplied in Supplement S1, bounded exact validation records in Supplement S2, and the clean generation scripts, environment record, and checksums in Supplement S3.

The publication reproducibility package regenerates the support tables and figures and reruns the bounded semantic checks without requiring the historical development archive. These computational products therefore provide independent bounded verification of the presentation and implementation, but no finite dataset is used to infer an infinite-family statement.

\section{Discussion and scope}
The classification concerns ordinary orthogonal $2\times2k$ maps with a six-bit period-two formal assignment and strict total orders of labelled faces. No symmetry quotient is taken, and distinct fold sequences that end in the same labelled total order are not distinguished. The theorem therefore does not classify arbitrary-period assignments, other map geometries, fold-sequence multiplicity, or unlabelled symmetry orbits.

At the state-set level, the two constant-row regimes have genuinely different structural descriptions: $MMMMVV$ is controlled by a boundary-or-diagonal compatibility condition, while $MVMMMM$ is controlled by equality of the block coordinate $B_R$. Multiplicity one converts those compatibility sets directly into exact labelled fibres, which is why the count formulas can remain secondary to the structural theorem. The separate $k=1$ proposition reflects a physical change in the crease set rather than an exception obtained by substituting $k=1$ into the generic two-parity argument.

The supplementary computation provides exact bounded validation checks, literal classification tables, and a clean reproduction path for the supplementary figures and tables. These finite products are useful for detecting transcription or implementation errors, but they do not enlarge the scope of the symbolic theorems.

\section{Conclusion}
The period-two $2\times2k$ family admits an exact labelled state-set description built from a local obstruction, exact transports, a pleat singleton mechanism, and two canonical constant-row compatibility mechanisms. The contrast between the constant-row branches is structural: $MMMMVV$ allows boundary-or-diagonal pivot pairs, whereas $MVMMMM$ allows exactly equal block-coordinate pairs. Compatible ordered pivot pairs parameterize labelled states with multiplicity one, and the formulas $6k-5$ and $4k-3$ follow by counting those fibres. The width-two boundary has its own empty-or-singleton criterion because its physical crease set omits two formal period bits.

\appendix
\section{Row-Local Normal Forms and the Local/Global Scope Boundary}\label{appendix-a.-row-local-normal-forms-and-the-localglobal-scope-boundary}

This appendix proves Proposition 6.1 using only row-local directed
precedences and same-target vertical noncrossing.

\subsection{Adjacent-span
trichotomy}\label{a.1.-adjacent-span-trichotomy}

For effective sign \(+\), put \(J_i=(O_i,D_i)\). For adjacent
\(J_i,J_{i+1}\), the directed inequalities are \(O_i\prec D_i\),
\(O_{i+1}\prec D_{i+1}\), and \(O_{i+1}\prec D_i\). Of the six total
orders compatible with the first two inequalities, one disjoint order
contradicts \(O_{i+1}\prec D_i\) and the two alternating orders are
forbidden by Butterfly noncrossing. The only possibilities are \[
\mathsf{L}_i: O_{i+1}\prec O_i\prec D_i\prec D_{i+1},
\] \[
\mathsf{R}_i: O_i\prec O_{i+1}\prec D_{i+1}\prec D_i,
\] \[
\mathsf{C}_i: O_{i+1}\prec D_{i+1}\prec O_i\prec D_i.
\]

\subsection{Relation-word normal
form}\label{a.2.-relation-word-normal-form}

Let \(K_i=(O_{i+1},D_i)\) be the even-seam span. For consecutive
relation symbols, each of
\((\mathsf R,\mathsf L),(\mathsf R,\mathsf C),(\mathsf C,\mathsf L),(\mathsf C,\mathsf C)\)
forces strict alternation of \(K_i\) and \(K_{i+1}\) and is therefore
impossible. Hence every admissible relation word is exactly \[
\mathsf L^*\mathsf R^*\qquad\text{or}\qquad \mathsf L^*\mathsf C\mathsf R^*.
\]

A relation word
\(\mathsf L_1\cdots \mathsf L_{s-1}\mathsf R_s\cdots \mathsf R_{k-1}\)
forces two nested chains around \(O_s,D_s\). At an interior junction,
putting \(O_{s+1}\) before \(D_{s-1}\) would make \(K_{s-1}\) and
\(K_s\) alternate, so \(D_{s-1}\prec O_{s+1}\). This fixes the entire
order as \(R^+(k,2s-1)\).

A relation word
\(\mathsf L_1\cdots \mathsf L_{s-1}\mathsf C_s\mathsf R_{s+1}\cdots \mathsf R_{k-1}\)
has the right block wholly before the left block because \(\mathsf C_s\)
gives \(D_{s+1}\prec O_s\). This fixes the order as \(R^+(k,2s)\).

Thus every row-local \(+\) order is one of the explicit normal forms.

\subsection{Sufficiency of the normal
forms}\label{a.3.-sufficiency-of-the-normal-forms}

In each odd-pivot normal form, all \(O_i\prec D_i\) and
\(O_{i+1}\prec D_i\) inequalities are immediate from the two monotone
blocks around the central pair. The \(J\) spans form two nested chains,
mutually disjoint except for the central containing span; the \(K\)
spans similarly form two nested chains. Hence every required same-target
pair is noncrossing.

For an even pivot, the right block precedes the left block. The directed
inequalities again hold, including the central even-seam inequality. The
\(J\) spans split into two nested chains separated by the central
\(\mathsf C\) relation. The span \(K_s\) contains the whole order, while
the remaining \(K\) spans form two nested chains inside it. Thus all
row-local constraints hold.

Therefore \[
R_{\mathrm{loc}}^+(k)=\{R^+(k,p):1\le p\le2k-1\}.
\] Literal reversal reverses every directed inequality and preserves
endpoint alternation, giving \[
R_{\mathrm{loc}}^-(k)=\{R^-(k,p):1\le p\le2k-1\}.
\]

\subsection{Pivot uniqueness}\label{a.4.-pivot-uniqueness}

For \(p=2s-1\), the first and last faces are \(O_s\) and \(D_s\), the
wings of seam \(2s-1\). For \(p=2s\), they are \(O_{s+1}\) and \(D_s\),
the wings of seam \(2s\). Hence the extreme pair identifies one physical
seam and therefore one pivot. Reversal exchanges first and last but
preserves the same seam.

\subsection{Local/global firewall}\label{a.5.-localglobal-firewall}

Every global state restricts to a row order satisfying the row-local
directed and Butterfly constraints. Therefore \[
R_{\mathrm{glob}}^{P,r}(k)\subseteq R_{\mathrm{loc}}^\sigma(k).
\] No reverse inclusion follows from the argument above. In particular,
the row-local count \(2k-1\) does not imply that all \(2k-1\) pivots
occur globally for an arbitrary period word. Global realization is
supplied only by the downstream canonical compatibility theorem in a
family where it has actually been proved.

\begin{center}\rule{0.5\linewidth}{0.5pt}\end{center}

\section{Exact labelled symmetry maps}\label{appendix-b.-exact-labelled-symmetry-maps}

This appendix gives the literal semantic verification behind Proposition
4.3. Pivot coordinates are treated only in B.6, after the
row-normal-form pivot has been defined in Section 6.

\subsection{\texorpdfstring{Complement
\(C\)}{Complement C}}\label{b.1.-complement-c}

Complement fixes every face and crease while replacing every M/V label
by its opposite. Checkerboard orientation and target classes are
unchanged, so every primitive directed precedence reverses. Literal
reversal of the total order reverses every strict inequality and
preserves strict endpoint alternation. Hence \[
T_C(\Pi)=\operatorname{rev}(\Pi)
\] is an exact labelled state bijection.

\subsection{\texorpdfstring{Row swap
\(R\)}{Row swap R}}\label{b.2.-row-swap-r}

The face map \(\phi_R\) exchanges \(A_c\) and \(B_c\), maps \(U_i\) to
\(L_i\) and conversely, and fixes \(H_c\). It flips checkerboard
light/dark orientation while preserving the M/V labels. Seam index, and
therefore the \(E/W\) target class, is unchanged. Thus primitive
precedences reverse and \[
T_R(\Pi)=\operatorname{rev}(\phi_R(\Pi)).
\] Incidence, disjointness, target class, and strict-alternation status
are preserved.

\subsection{\texorpdfstring{Horizontal reflection
\(H\)}{Horizontal reflection H}}\label{b.3.-horizontal-reflection-h}

The map \(\phi_H\) sends columns by \(c\mapsto n+1-c\), so \[
H_c\mapsto H_{n+1-c},\qquad
U_i\mapsto U_{n-i},\qquad
L_i\mapsto L_{n-i}.
\] Because \(n\) is even, the checkerboard orientation of every face
flips. Horizontal-crease period parity therefore swaps \(h_0\) and
\(h_1\). By contrast, \(n-i\) has the same parity as \(i\), so the \(u\)
and \(l\) period slots and the \(E/W\) target classes are preserved. The
precedence reversal caused by the checkerboard flip gives \[
T_H(\Pi)=\operatorname{rev}(\phi_H(\Pi)).
\] A physical seam \(p\) between columns \(p\) and \(p+1\) maps to the
seam between columns \(n-p\) and \(n-p+1\); therefore \[
p\longmapsto n-p.
\]

\subsection{\texorpdfstring{The compositions \(CR\) and
\(CH\)}{The compositions CR and CH}}\label{b.4.-the-compositions-cr-and-ch}

For \(CR\), row swap reverses the directed precedences and complement
reverses them again. The two effects cancel, hence \[
T_{CR}(\Pi)=\phi_R(\Pi).
\] For \(CH\), horizontal reflection and complement likewise contribute
two reversals, so \[
T_{CH}(\Pi)=\phi_H(\Pi).
\] In both cases the same incidence, target-class, and
Butterfly-preservation arguments apply.

\subsection{Exact two-direction state-set
equivalence}\label{b.5.-exact-two-direction-state-set-equivalence}

Each face transformation is bijective and preserves crease incidence and
disjointness. Corresponding creases have the same folded target class.
Bijective relabelling preserves endpoint positions up to renaming, and
literal order reversal is an order anti-isomorphism; both preserve
whether two disjoint same-target crease intervals strictly alternate.
Thus every semantic clause holds for \(\Pi\) exactly when it holds for
\(T_g(\Pi)\). Since every listed transformation is involutive, \[
\Pi\in F(P,k)\iff T_g(\Pi)\in F(gP,k),
\] and therefore \[
F(gP,k)=T_g(F(P,k)).
\] This is equality of literal labelled state sets, not a symmetry
quotient.

\subsection{Pivot coordinates after the row-normal-form
definition}\label{b.6.-pivot-coordinates-after-the-row-normal-form-definition}

Let \(R^\sigma(k,p)\) be the row-local normal form of Proposition 6.1,
where \(p\) is the unique physical seam joining the extreme faces of the
row order. Then \[
T_C(R^\sigma(k,p))=R^{-\sigma}(k,p),
\] so complement leaves the numerical pivot unchanged.

Row swap exchanges physical rows without reflecting seam indices. Thus
on an ordered upper/lower pair, \[
(a,b)\longmapsto(b,a)
\] under both \(R\) and \(CR\).

Horizontal reflection sends the physical seam \(p\) to \(n-p\). The row
normal forms satisfy \[
T_H(R^\sigma(k,p))=R^\sigma(k,n-p),
\] \[
T_{CH}(R^\sigma(k,p))=R^{-\sigma}(k,n-p),
\] and therefore \[
(a,b)\longmapsto(n-a,n-b)
\] under \(H\) and \(CH\). These identities transport row-local pivots;
they do not assert that an arbitrary local pivot is globally realizable.

\begin{center}\rule{0.5\linewidth}{0.5pt}\end{center}

\section{Full Validity of the Canonical Pleat Orders}\label{appendix-c.-full-validity-of-the-canonical-pleat-orders}

The Hamiltonian arguments in Theorem 5.1 prove uniqueness among
precedence-compatible orders. This appendix proves existence by checking
every Butterfly family at arbitrary separation.

\subsection{\texorpdfstring{The order
\(\Omega_+(k)\)}{The order \textbackslash Omega\_+(k)}}\label{c.1.-the-order-omega_k}

For \(1\le q\le k\) the positions are \[
\operatorname{pos}(A_{2q-1})=4q-3,\quad
\operatorname{pos}(B_{2q-1})=4q-2,
\] \[
\operatorname{pos}(B_{2q})=4q-1,\quad
\operatorname{pos}(A_{2q})=4q.
\] All primitive M/V precedences point forward in this order. The sorted
endpoint intervals are \[
I(H_{2q-1})=[4q-3,4q-2],\qquad I(H_{2q})=[4q-1,4q],
\] \[
I(U_{2q-1})=[4q-3,4q],\qquad I(U_{2q})=[4q,4q+1],
\] \[
I(L_{2q-1})=[4q-2,4q-1],\qquad I(L_{2q})=[4q-1,4q+2],
\] where the even-seam formulas run only to \(q=k-1\).

For \(H/H\), distinct column intervals are strictly disjoint. For
same-parity \(U/U\) and \(L/L\), if \(q<r\) the right endpoint of the
\(q\) interval is strictly before the left endpoint of the \(r\)
interval; hence these pairs are disjoint.

For the cross-row \(U/L\) class there is no distance cutoff. For odd
seams \(i=2q-1\), \(j=2r-1\), \[
I(U_i)=[4q-3,4q],\qquad I(L_j)=[4r-2,4r-1].
\] If \(r=q\), the lower interval is strictly nested in the upper one;
if \(r<q\) it is strictly before it; if \(r>q\) it is strictly after it.
For even seams \(i=2q\), \(j=2r\), \[
I(U_i)=[4q,4q+1],\qquad I(L_j)=[4r-1,4r+2],
\] and the same trichotomy holds, with the upper interval nested in the
lower one when \(r=q\). Thus every same-target pair is nested or
disjoint for arbitrary \(q,r\).

\subsection{\texorpdfstring{The order
\(\Omega_-(k)\)}{The order \textbackslash Omega\_-(k)}}\label{c.2.-the-order-omega_-k}

Here \[
\operatorname{pos}(A_c)=c,\qquad
\operatorname{pos}(B_c)=2n-c+1.
\] The primitive directed relations are \(A_i\prec A_{i+1}\),
\(B_{i+1}\prec B_i\), and \(A_c\prec B_c\), all of which point forward.

For horizontal creases, \[
I(H_c)=[c,2n-c+1].
\] If \(c<d\), then \[
c<d<2n-d+1<2n-c+1,
\] so \(I(H_d)\) is strictly nested in \(I(H_c)\).

Upper vertical intervals are \(I(U_i)=[i,i+1]\). Distinct same-target
seams have equal parity and hence differ by at least two, so these
intervals are disjoint. Lower vertical intervals are \[
I(L_i)=[2n-i,2n-i+1]
\] and are likewise disjoint within a target parity.

Finally, \[
\max I(U_i)\le n<n+1\le\min I(L_j)
\] for all \(i,j\), so every cross-row \(U/L\) pair is separated, even
before target parity is imposed.

\subsection{Boundary and strictness
check}\label{c.3.-boundary-and-strictness-check}

The Butterfly predicate uses strict alternation. Genuine tested pairs
have disjoint wing sets, so their four endpoint positions are distinct.
At the left and right boundaries all displayed interval formulas stay
within \(1,\ldots,4k\); even seam formulas stop at \(n-2\). Therefore no
endpoint degeneration or omitted long-range case occurs. The families
\(H/H\), \(U/U\), \(L/L\), and \(U/L\) exhaust all same-target disjoint
crease pairs because horizontal target \(S\) is distinct from vertical
targets \(E,W\).

Hence both canonical pleat orders satisfy every global semantic
condition for every \(k\ge2\).

\begin{center}\rule{0.5\linewidth}{0.5pt}\end{center}

\section{\texorpdfstring{Technical Closure for
\(MMMMVV\)}{Technical Closure for MMMMVV}}\label{appendix-d.-technical-closure-for-mmmmvv}

This appendix supplies the detailed necessity, existence, full-validity,
and uniqueness bridges used in Theorem 7.1.

\subsection{Row-normal-form column words and horizontal
actions}\label{d.1.-row-normal-form-column-words-and-horizontal-actions}

Put \(o_r=2r-1\), \(e_r=2r\) for \(1\le r\le k\). The plus-normal-form
column word is \[
W_{2s-1}=(o_s,o_{s-1},\ldots,o_1,e_1,\ldots,e_{s-1},o_{s+1},\ldots,o_k,e_k,\ldots,e_{s+1},e_s),
\] \[
W_{2s}=(o_{s+1},\ldots,o_k,e_k,\ldots,e_{s+1},o_s,\ldots,o_1,e_1,\ldots,e_s).
\] For \(MMMMVV\), horizontal crease \(H_{o_r}\) has opener \(A_{o_r}\)
and closer \(B_{o_r}\), while \(H_{e_r}\) has opener \(B_{e_r}\) and
closer \(A_{e_r}\).

\subsection{Necessity: the four-edge
cycle}\label{d.2.-necessity-the-four-edge-cycle}

Define \(X_p(i,j)\) by \(e_j\) preceding \(o_i\) in \(W_p\). Direct
inspection gives \[
X_{2s-1}(i,j)\iff (i>s\ \text{and}\ j<s),
\] \[
X_{2s}(i,j)\iff (i\le s\ \text{and}\ j>s).
\] If \(1<a<m\) and \(a\ne b\), choose \[
(i,j)=(s+1,s-1)\quad\text{when }a=2s-1,
\] \[
(i,j)=(s,s+1)\quad\text{when }a=2s.
\] The legal ranges are \(2\le s\le k-1\) in the odd case and
\(1\le s\le k-1\) in the even case. The following complete case table
verifies \(X_a=1\) and \(X_b=0\):

\begin{longtable}[]{@{}
  >{\raggedright\arraybackslash}p{(\columnwidth - 6\tabcolsep) * \real{0.2500}}
  >{\raggedright\arraybackslash}p{(\columnwidth - 6\tabcolsep) * \real{0.2500}}
  >{\raggedright\arraybackslash}p{(\columnwidth - 6\tabcolsep) * \real{0.2500}}
  >{\raggedright\arraybackslash}p{(\columnwidth - 6\tabcolsep) * \real{0.2500}}@{}}
\toprule\noalign{}
\begin{minipage}[b]{\linewidth}\raggedright
parity of \((a,b)\)
\end{minipage} & \begin{minipage}[b]{\linewidth}\raggedright
condition
\end{minipage} & \begin{minipage}[b]{\linewidth}\raggedright
witness
\end{minipage} & \begin{minipage}[b]{\linewidth}\raggedright
reason \(X_b=0\)
\end{minipage} \\
\midrule\noalign{}
\endhead
\bottomrule\noalign{}
\endlastfoot
odd/odd & \(a<b\) & \((s+1,s-1)\) & \(X_b\) would require \(i>t\), but
\(t\ge s+1=i\) \\
odd/odd & \(a>b\) & \((s+1,s-1)\) & \(X_b\) would require \(j<t\), but
\(t\le s-1=j\) \\
odd/even & any \(b\ne a\) & \((s+1,s-1)\) & if \(b=2t\), then \(X_b=1\)
would require \(s+1\le t\) and \(s-1>t\), simultaneously \\
even/odd & any \(b\ne a\) & \((s,s+1)\) & if \(b=2t-1\), then \(X_b=1\)
would require \(s>t\) and \(s+1<t\), simultaneously \\
even/even & \(a<b\) & \((s,s+1)\) & \(X_b\) would require \(j>t\), but
\(t\ge s+1=j\) \\
even/even & \(a>b\) & \((s,s+1)\) & \(X_b\) would require \(i\le t\),
but \(t\le s-1<i=s\) \\
\end{longtable}

Set \(o=2i-1\) and \(e=2j\). In any merge with pivots \((a,b)\), \[
A_e\prec A_o
\] comes from \(X_a=1\) in the upper row. Since \(o\) is odd and \(e\)
even, horizontal orientation gives \[
A_o\prec B_o,\qquad B_e\prec A_e.
\] Since \(X_b=0\), the strict lower row gives \(B_o\prec B_e\). Thus \[
A_e\prec A_o\prec B_o\prec B_e\prec A_e,
\] impossible in a total order. At \(k=2\), the only interior upper
pivot is \(a=2\), the witness is \((i,j)=(1,2)\), and the same cycle is
\(A_4\prec A_1\prec B_1\prec B_4\prec A_4\).

\subsection{Five explicit existence
families}\label{d.3.-five-explicit-existence-families}

For a column \(c\), define \[
Q_c=(A_c,B_c)\quad(c\text{ odd}),\qquad
Q_c=(B_c,A_c)\quad(c\text{ even}).
\] Empty ranges are omitted.

If \(W_p=(c_1,\ldots,c_{2k})\), define the diagonal construction \[
D(k,p)=Q_{c_1}Q_{c_2}\cdots Q_{c_{2k}},
\] which has pivots \((p,p)\).

For \(1\le s\le k\) define \[
\begin{aligned}
L_{\rm odd}(k,s)=&\ A_{o_1}\cdots A_{o_s}\mid B_{o_s}\cdots B_{o_1}\mid B_{e_1}\cdots B_{e_{s-1}}\\
&\mid Q_{o_{s+1}}\cdots Q_{o_k}\mid Q_{e_k}\cdots Q_{e_s}\mid A_{e_{s-1}}\cdots A_{e_1},
\end{aligned}
\] with pivots \((1,2s-1)\), and for \(1\le s\le k-1\) define \[
\begin{aligned}
L_{\rm even}(k,s)=&\ A_{o_1}\cdots A_{o_{s+1}}\mid B_{o_{s+1}}\mid Q_{o_{s+2}}\cdots Q_{o_k}\\
&\mid Q_{e_k}\cdots Q_{e_{s+1}}\mid B_{o_s}\cdots B_{o_1}\mid B_{e_1}\cdots B_{e_s}\mid A_{e_s}\cdots A_{e_1},
\end{aligned}
\] with pivots \((1,2s)\).

For \(1\le s\le k\) define \[
\begin{aligned}
R_{\rm odd}(k,s)=&\ A_{o_k}\cdots A_{o_s}\mid B_{o_s}\mid Q_{o_{s-1}}\cdots Q_{o_1}\mid Q_{e_1}\cdots Q_{e_{s-1}}\\
&\mid B_{o_{s+1}}\cdots B_{o_k}\mid B_{e_k}\cdots B_{e_s}\mid A_{e_s}\cdots A_{e_k},
\end{aligned}
\] with pivots \((m,2s-1)\), and for \(1\le s\le k-1\) define \[
\begin{aligned}
R_{\rm even}(k,s)=&\ A_{o_k}\cdots A_{o_{s+1}}\mid B_{o_{s+1}}\cdots B_{o_k}\mid B_{e_k}\cdots B_{e_{s+1}}\\
&\mid Q_{o_s}\cdots Q_{o_1}\mid Q_{e_1}\cdots Q_{e_s}\mid A_{e_{s+1}}\cdots A_{e_k},
\end{aligned}
\] with pivots \((m,2s)\).

The index ranges partition the four label classes, so every displayed
word is a strict total order of all \(4k\) labels. Deleting all \(B\)
labels or all \(A\) labels gives exactly the claimed row-normal-form
words. The true formula overlaps are \((1,1)\) and \((m,m)\), where the
diagonal and boundary words are literally identical. The pairs \((1,m)\)
and \((m,1)\) are ordinary boundary members of \(L_{\rm odd}(k,k)\) and
\(R_{\rm odd}(k,1)\). All formulas remain literal at \(k=2\) under the
empty-range convention.

\subsection{Horizontal stack validity of the
constructions}\label{d.4.-horizontal-stack-validity-of-the-constructions}

For \(MMMMVV\), the horizontal actions are

\[
A_{o_r}=\operatorname{push}(o_r),\qquad B_{o_r}=\operatorname{pop}(o_r),\qquad
B_{e_r}=\operatorname{push}(e_r),\qquad A_{e_r}=\operatorname{pop}(e_r).
\]

Each \(Q_c\) is therefore an immediate push/pop pair above the
pre-existing stack. The five construction families can be checked
literally.

\begin{itemize}
\tightlist
\item
  \textbf{Diagonal \(D(k,p)\).} Every \(Q_{c_j}\) pushes and immediately
  pops \(c_j\), so the stack is empty before and after every block.
\item
  \textbf{Left odd \(L_{\rm odd}(k,s)\).} The initial
  \(A_{o_1},\ldots,A_{o_s}\) pushes the odd core \((o_1,\ldots,o_s)\),
  and \(B_{o_s},\ldots,B_{o_1}\) pops it. Then
  \(B_{e_1},\ldots,B_{e_{s-1}}\) pushes the even core
  \((e_1,\ldots,e_{s-1})\). All later \(Q\) blocks act above that core,
  and \(A_{e_{s-1}},\ldots,A_{e_1}\) pops it.
\item
  \textbf{Left even \(L_{\rm even}(k,s)\).} The first odd run pushes
  \((o_1,\ldots,o_{s+1})\); \(B_{o_{s+1}}\) pops its top, leaving
  \((o_1,\ldots,o_s)\). The intervening odd and even \(Q\) blocks act
  above that core. Then \(B_{o_s},\ldots,B_{o_1}\) empties the odd core,
  \(B_{e_1},\ldots,B_{e_s}\) pushes a fresh even core, and
  \(A_{e_s},\ldots,A_{e_1}\) pops it.
\item
  \textbf{Right odd \(R_{\rm odd}(k,s)\).} The first run pushes
  \((o_k,o_{k-1},\ldots,o_s)\); \(B_{o_s}\) removes the top. The \(Q\)
  blocks act above the residual odd core \((o_k,\ldots,o_{s+1})\), which
  is then popped by \(B_{o_{s+1}},\ldots,B_{o_k}\). Finally
  \(B_{e_k},\ldots,B_{e_s}\) pushes the even core \((e_k,\ldots,e_s)\)
  and \(A_{e_s},\ldots,A_{e_k}\) pops it.
\item
  \textbf{Right even \(R_{\rm even}(k,s)\).} The initial odd core
  \((o_k,\ldots,o_{s+1})\) is pushed and then popped completely. The run
  \(B_{e_k},\ldots,B_{e_{s+1}}\) pushes an even core; all \(Q\) blocks
  act above it; the final run \(A_{e_{s+1}},\ldots,A_{e_k}\) pops it.
\end{itemize}

All empty endpoint ranges are literal, so the same verification covers
\(s=1\), \(s=k\), and \(s=k-1\) where those parameters are legal. In
every family there is no underflow, every closer matches the current
top, and the final stack is empty. Consequently every horizontal
directed relation holds and every two horizontal intervals are disjoint
or nested, never strictly alternating.

\subsection{\texorpdfstring{Remaining \(U/L\) full-validity
obligations}{Remaining U/L full-validity obligations}}\label{d.5.-remaining-ul-full-validity-obligations}

Proposition 6.1 already handles all directed vertical relations and
same-row equal-target pairs. H-stack validity handles the horizontal
system. Therefore only $U_i/L_j$ with equal seam parity remain.

For the diagonal family, each column is a contiguous block $Q_c$ in
the row-normal-form column order $W_p$. If $i\ne j$ have equal parity,
then the adjacent-column pairs $\{i,i+1\}$ and $\{j,j+1\}$ are disjoint.
Strict alternation of the actual $U/L$ endpoints would force alternation
of those column pairs in $W_p$, contradicting row-local validity from
Proposition 6.1. For $i=j$, the two adjacent column blocks occur with
opposite internal orientations, and direct inspection shows one vertical
interval nested in the other.

For a boundary family fix the upper $A$ chain. Let $\rho(c)$ be the
position of $A_c$, and put
\[
\sigma(c)=\#\{A_d:A_d\prec B_c\}.
\]

\begin{lemma}[Slots and cross-row alternation]\label{lem:per2-slots}
Suppose an upper seam has endpoint ranks $\alpha<\beta$ in the upper
$A$ chain and a lower seam has endpoint slots $x\le y$. The two seams
strictly alternate if and only if exactly one of $x,y$ belongs to the
integer interval $[\alpha,\beta-1]$. In particular, alternation is
impossible if $x=y$, or if $x+y=\alpha+\beta-1$.
\end{lemma}
\begin{proof}
A lower face lies strictly between the upper endpoints exactly when at
least $\alpha$, but fewer than $\beta$, upper faces precede it, i.e. its
slot belongs to $[\alpha,\beta-1]$. Two intervals with four distinct
endpoints alternate exactly when precisely one endpoint of the second
interval lies inside the first. This gives the criterion and the same-slot
assertion. If $x+y=\alpha+\beta-1$ and $x\ge\alpha$, then
$y\le\beta-1$, so both slots lie inside. If $x<\alpha$, integrality gives
$x\le\alpha-1$ and hence $y\ge\beta$, so the lower endpoints lie on the
two exterior sides. Neither case gives exactly one interior endpoint.
\end{proof}

For upper chain $W_1$,
\[
\rho_1(o_r)=r,\qquad \rho_1(e_r)=2k-r+1,
\]
so
\[
S_r^o=[r,2k-r]\quad(1\le r\le k),\qquad
S_r^e=[r+1,2k-r]\quad(1\le r\le k-1).
\]
For upper chain $W_m$,
\[
\rho_m(o_r)=k-r+1,\qquad \rho_m(e_r)=k+r,
\]
so
\[
T_r^o=[k-r+1,k+r-1]\quad(1\le r\le k),\qquad
T_r^e=[k-r,k+r-1]\quad(1\le r\le k-1).
\]
The sums of the two endpoints in these four shell families are,
respectively, $2k$, $2k+1$, $2k$, and $2k-1$.

The required \(\sigma\) values are as follows; the three columns
correspond to \(t<s\), \(t=s\), and \(t>s\).

\begin{longtable}[]{@{}llrrr@{}}
\toprule\noalign{}
construction & column & \(t<s\) & \(t=s\) & \(t>s\) \\
\midrule\noalign{}
\endhead
\bottomrule\noalign{}
\endlastfoot
\(L_{\rm odd}(k,s)\) & \(\sigma(o_t)\) & \(s\) & \(s\) & \(t\) \\
& \(\sigma(e_t)\) & \(s\) & \(2k-s\) & \(2k-t\) \\
\(L_{\rm even}(k,s)\) & \(\sigma(o_t)\) & \(2k-s\) & \(2k-s\) & \(t\) \\
& \(\sigma(e_t)\) & \(2k-s\) & \(2k-s\) & \(2k-t\) \\
\(R_{\rm odd}(k,s)\) & \(\sigma(o_t)\) & \(k-t+1\) & \(k-s+1\) &
\(k+s-1\) \\
& \(\sigma(e_t)\) & \(k+t-1\) & \(k+s-1\) & \(k+s-1\) \\
\(R_{\rm even}(k,s)\) & \(\sigma(o_t)\) & \(k-t+1\) & \(k-s+1\) &
\(k-s\) \\
& \(\sigma(e_t)\) & \(k+t-1\) & \(k+s-1\) & \(k-s\) \\
\end{longtable}

These values are obtained directly by counting preceding \(A\) labels in
the displayed construction words. They give the following complete
lower-seam classification.

\begin{itemize}
\tightlist
\item
  In \(L_{\rm odd}\), odd seam \(2t-1\) is a point for \(t<s\), is
  \(S_s^o\) for \(t=s\), and is \(S_t^o\) for \(t>s\); even seam \(2t\)
  is a point for \(t<s\), is \(S_s^e\) for \(t=s\), and is \(S_t^e\) for
  \(t>s\).
\item
  In \(L_{\rm even}\), odd seam \(2t-1\) is a point for \(t\le s\) and
  is \(S_t^o\) for \(t>s\); even seam \(2t\) is a point for \(t<s\), is
  \(S_s^e\) for \(t=s\), and is \(S_t^e\) for \(t>s\).
\item
  In \(R_{\rm odd}\), odd seam \(2t-1\) is \(T_t^o\) for \(t\le s\) and
  a point for \(t>s\); even seam \(2t\) is \(T_t^e\) for \(t<s\) and a
  point for \(t\ge s\).
\item
  In \(R_{\rm even}\), odd seam \(2t-1\) is \(T_t^o\) for \(t\le s\) and
  a point for \(t>s\); even seam \(2t\) is \(T_t^e\) for \(t\le s\) and
  a point for \(t>s\).
\end{itemize}

Thus every non-point lower seam is exactly a member of the corresponding
upper shell family. For an arbitrary upper seam of the same parity, its
shell and the two lower endpoint slots belong to the same one of the four
families above and therefore have the same endpoint sum. If the upper
shell is $[\alpha,\beta-1]$, the lower slots satisfy
$x+y=\alpha+\beta-1$, so Lemma~\ref{lem:per2-slots} excludes strict
alternation. Point seams are excluded by the same lemma. This proves all
remaining $U_i/L_j$ obligations for arbitrary separation of the seam
indices. The argument uses the constant endpoint sum; numerical
containment of slot intervals by itself is not used.

\subsection{H-local legality is not
H-completability}\label{d.6.-h-local-legality-is-not-h-completability}

For fixed row chains, represent a partial H merge by consumed row-prefix
lengths and the stack \(S\) of open columns. A next row head is
H-locally legal if it obeys one push/pop step. It is H-completable only
if the child state admits a full continuation ending with every face
consumed and the stack empty.

These notions differ. At \(k=2\) with \((a,b)=(1,1)\), after prefix
\(A_1\) both \(A_3\) and \(B_1\) are locally legal, but only \(B_1\)
admits completion. Therefore the uniqueness proof never uses the false
rule ``every locally legal next event is globally forced.''

\begin{lemma}[Diagonal fixed-row uniqueness]\label{lem:per2-diagonal}
For the canonical assignment $\texttt{MMMMVV}$ and every
$1\le p\le m$, the two row orders with pivots $(p,p)$ have exactly one
complete H-valid merge, namely $D(k,p)$.
\end{lemma}
\begin{proof}
Write their common column order as $W_p=(c_1,\ldots,c_n)$.
Assume inductively that the first $r-1$ columns have been processed as
$Q_{c_1}\cdots Q_{c_{r-1}}$. The horizontal stack is empty and both row
heads have column $c_r$. The opener must be emitted first: it is
$A_{c_r}$ for odd $c_r$, and $B_{c_r}$ for even $c_r$.

The other row is now waiting at its matching closer. Suppose the row
that emitted the opener advances instead to its next column $d$.
If that next event is a closer, its matching opener has not appeared,
since it occurs after $c_r$ in the blocked opposite row. If it is an
opener, its label $d$ is pushed above $c_r$. In any complete LIFO trace,
$d$ would then have to close before $c_r$, whereas the opposite row
requires the closer of $c_r$ before the closer of $d$. Both alternatives
are impossible. If this row is exhausted, there is no alternative event.
Hence the matching closer follows immediately, producing $Q_{c_r}$,
emptying the stack, and restoring the induction hypothesis.
The resulting word is $D(k,p)$, which is H-valid by its immediate
push--pop blocks. The proof includes $p=1,m$ as well as interior pivots.
\end{proof}

\subsection{Projected determinism and first
disagreement}\label{d.7.-projected-determinism-and-first-disagreement}

In any complete H-valid merge, deleting all events of one column parity
leaves a legal stack trace. For the odd projection, the upper row fixes
the opener stream and the lower row fixes the closer stream. For the
even projection these roles are exchanged. A fixed opener stream and a
fixed closer stream have at most one complete legal stack trace:
if the next requested closer has not opened, the opener stream is forced
through that label; once it is the top label, closing it is necessary,
since another push would bury the fixed next closer. A requested closer
that is already buried cannot be completed. Thus both parity projections
are uniquely determined whenever they exist.

For a boundary pivot pair, let $T$ be the explicit H-valid construction
from D.3--D.4. Suppose a different complete H-valid merge $T'$ exists,
and consider their first divergence. Their common prefix is a prefix
of the known word $T$. An upper-odd/lower-odd divergence contradicts
uniqueness of the odd projection, and an upper-even/lower-even divergence
contradicts uniqueness of the even projection. An upper-even/lower-odd
divergence would require two distinct closers to be simultaneously legal
at one stack top. The only remaining possibility is an upper-odd opener
versus a lower-even opener.

\subsection{Opener windows on the explicit
constructions}\label{d.8.-corrected-cross-parity-windows}

The following table lists all prefixes of the four explicit boundary
constructions at which both row heads are openers. In the stack column,
entries are written from bottom to top.
\[
\begin{array}{c|c|c|c}
T&\text{range}&\text{stack}&\text{row heads}\\\hline
L_{\rm odd}(k,s)&2\le s\le k-1,\ 1\le t\le s-1
 &(e_1,\ldots,e_{t-1})&(A_{o_{s+1}},B_{e_t})\\
L_{\rm even}(k,s)&\text{none}&-&-\\
R_{\rm odd}(k,s)&\text{none}&-&-\\
R_{\rm even}(k,s)&1\le s\le k-1,\ s+1\le t\le k
 &(e_k,e_{k-1},\ldots,e_{t+1})&(A_{o_s},B_{e_t}).
\end{array}
\]
In every listed window, $T$ chooses the lower-even opener.
For $L_{\rm odd}$, the initial odd core is opened and closed before the
lower even core $e_1,\ldots,e_{s-1}$ is opened. During this last phase
the next upper face is $A_{o_{s+1}}$ exactly when $s<k$. The range is
empty for $s=1$, and for $s=k$ the upper head is already an even closer.
For $R_{\rm even}$, the initial odd core is likewise emptied; the lower
even core is then opened from $e_k$ down to $e_{s+1}$ while the upper
head remains $A_{o_s}$, giving the second displayed range.
For $L_{\rm even}$ and $R_{\rm odd}$, all upper odd openers have appeared
before the first lower even opener is reached. Immediate $Q_c$ blocks
and the final closing phases create no further opener/opener windows.
This exhausts prefixes of the explicit constructions, not arbitrary
H-locally legal dead-end prefixes. That is exactly the scope required
by the first-divergence argument.

\subsection{Wrong-switch barriers and completion of
uniqueness}\label{d.9.-wrong-switch-barriers}

At a left-odd window, suppose $T'$ chooses $A_{o_{s+1}}$ instead of
$B_{e_t}$. Before the lower row can reach $B_{o_{s+1}}$, it must open
$e_t,\ldots,e_{s-1}$ above the newly opened odd label. Consequently
LIFO requires $A_{e_{s-1}}\prec B_{o_{s+1}}$. The fixed row orders and
horizontal precedence then give the impossible cycle
\[
 B_{o_{s+1}}
 \mathrel{\prec_B} B_{e_k}
 \mathrel{\prec_H} A_{e_k}
 \mathrel{\prec_A} A_{e_{s-1}}
 \mathrel{\prec_{\rm LIFO}} B_{o_{s+1}}.
\]
Here $\prec_A,\prec_B$ are the prescribed upper and lower row orders;
$\prec_H$ is horizontal opener-before-closer precedence.

At a right-even window, choosing $A_{o_s}$ forces the later lower opener
$e_{s+1}$ above it before $B_{o_s}$ can appear. LIFO requires
$A_{e_{s+1}}\prec B_{o_s}$, while the other constraints give
\[
 B_{o_s}
 \mathrel{\prec_B} B_{e_1}
 \mathrel{\prec_H} A_{e_1}
 \mathrel{\prec_A} A_{e_{s+1}}
 \mathrel{\prec_{\rm LIFO}} B_{o_s}.
\]
The index ranges in the window table ensure that all events used in
these cycles exist. Thus a first divergence from $T$ is impossible.
Each boundary pair therefore has at most one complete H-valid merge.
Lemma~\ref{lem:per2-diagonal} covers every diagonal pair. Together these
cases cover all pairs $a=1$, $a=m$, or $a=b$. The constructions give
existence and D.5 gives full validity, completing the multiplicity-one
claim without using a count formula.

\subsection{Global-state-to-pivot bridge and
multiplicity}\label{d.10.-global-state-to-pivot-bridge-and-multiplicity}

Every globally valid \(MMMMVV\) state obeys all row-local constraints.
Proposition 6.1 therefore assigns its upper and lower restrictions
unique pivots \((a,b)\). This is a coverage/parameterization statement
only; it does not itself assert compatibility or existence.

Necessity shows the recovered pair is allowed. Sections D.3--D.5 give
one fully valid state for each allowed pair. Lemma~\ref{lem:per2-diagonal}
gives fixed-pair uniqueness on the diagonal, while Sections D.7--D.9
give fixed-pair uniqueness for the boundary families. These cases exhaust
$a=1$, $a=m$, or $a=b$. Hence the global-state-to-pair map is a
bijection and each compatible pair has multiplicity one. Only after this bridge may the compatible-pair
cardinality be identified with the state count.

\begin{center}\rule{0.5\linewidth}{0.5pt}\end{center}

\section{\texorpdfstring{Reconstruction and
block-compatibility closure for
\(MVMMMM\)}{Reconstruction and block-compatibility closure for MVMMMM}}\label{appendix-e.-reconstruction-and-block-compatibility-closure-for-mvmmmm}

This appendix keeps the two logical ingredients of Theorem 7.2 separate:
fixed-row reconstruction/full validity, and the independent block
criterion.

\subsection{Fixed-row semantics}\label{e.1.-fixed-row-semantics}

For \(MVMMMM\), every horizontal relation is \[
A_c\prec B_c.
\] The upper row is \(R^+(k,a)\) and the lower row is \(R^-(k,b)\).
Hence the upper column sequence is the H-opener stream, while the lower
column sequence is the H-closer request stream.

\subsection{Fixed-row reconstruction: H/H noncrossing is
LIFO}\label{e.2.-fixed-row-reconstruction-hh-noncrossing-is-lifo}

With each opener before its matching closer, two horizontal intervals
cross exactly when an earlier-opened interval closes while a
later-opened distinct interval is still open. Pairwise \(H/H\) Butterfly
noncrossing is therefore equivalent to the usual LIFO stack rule.

Fix the two row chains and let \(c\) be the next requested lower label.
If the matching upper opener has not yet been emitted, no lower label
may bypass \(c\), so the upper chain is forced until that opener
appears. If \(c\) is the stack top, its closer is forced immediately;
pushing another opener would block the fixed next closer. If \(c\) is
already buried in the stack, no completion exists. Induction over the
lower events proves that an H-compatible merge is unique when it exists,
and that the forced reconstruction succeeds exactly when H compatibility
holds.

\subsection{Automatic remainder after horizontal
reconstruction}\label{e.3.-automatic-remainder-after-horizontal-reconstruction}

Proposition 6.1 supplies all upper and lower vertical directed
constraints and all same-row same-target vertical noncrossing. The H
stack supplies every horizontal precedence and every \(H/H\) noncrossing
condition. Since the horizontal target \(S\) is distinct from the
vertical targets \(E,W\), no mixed horizontal/vertical twin family
occurs. The only remaining class is \(U_i/L_j\) with equal seam parity.

For seam \(i\), orient the adjacent horizontal nesting as \[
A_{\alpha_i}\prec A_{\beta_i}\prec B_{\beta_i}\prec B_{\alpha_i},
\] where \((\alpha_i,\beta_i)=(i,i+1)\) for odd \(i\) and \((i+1,i)\)
for even \(i\). Thus \(U_i\) is the opener-side interval and \(L_i\) the
closer-side interval of the same nested H pair.

For distinct equal-parity seams \(i,j\), the two adjacent column pairs
are disjoint. Consider two nested H pairs \[
o_x<o_y<c_y<c_x,
\qquad
o_p<o_q<c_q<c_p,
\] with opener-side intervals \(U_e=[o_x,o_y]\), \(U_f=[o_p,o_q]\) and
closer-side intervals \(L_e=[c_y,c_x]\), \(L_f=[c_q,c_p]\). Suppose
\(U_e\) crosses \(L_f\).

If \[
c_q<o_x<c_p<o_y,
\] then \[
o_p<o_x<c_p<c_x,
\] so the horizontal intervals \([o_p,c_p]\) and \([o_x,c_x]\) cross,
contradicting horizontal laminarity. Hence the only possible alternating
order is \[
o_x<c_q<o_y<c_p.
\] If \(o_p<o_x\), horizontal laminarity of \([o_q,c_q]\) with
\([o_x,c_x]\) forces \(o_x<o_q\), giving \[
o_p<o_x<o_q<o_y,
\] so \(U_e\) crosses \(U_f\). If instead \(o_x<o_p\), laminarity of
\([o_p,c_p]\) with \([o_y,c_y]\) forces \(c_y<c_p\), while laminarity of
\([o_x,c_x]\) with \([o_p,c_p]\) forces \(c_p<c_x\). Together with
\(c_q<o_y<c_y\) this yields \[
c_q<c_y<c_p<c_x,
\] so \(L_e\) crosses \(L_f\). Therefore \[
U_i\text{ crosses }L_j
\Longrightarrow
(U_i\text{ crosses }U_j)\ \text{or}\ (L_i\text{ crosses }L_j).
\] Equal parity gives the same vertical target class, and Proposition
6.1 forbids both same-row alternatives. For \(i=j\), the upper interval
lies in the opener side of the nested adjacent H pair and the lower
interval in the closer side, so they are disjoint. Thus every
H-compatible merge is globally valid: \[
\operatorname{FullCompat}(k;a,b)
\iff
\operatorname{HCompat}(k;a,b).
\]

\subsection{Endpoint normal forms for the block
criterion}\label{e.4.-endpoint-normal-forms-for-the-block-criterion}

Let \(W_p\) denote the \(+\) row-normal-form column word for pivot
\(p\). Because the lower row is the reverse normal form, its closer
request stream is \[
\overleftarrow{W}_p:=\operatorname{rev}(W_p).
\] The singleton-block word is \[
W_1=(1,3,5,\ldots,2k-1,2k,2k-2,\ldots,2).
\] For \(1\le r\le k-1\), name the two pivots in block \(r\) by \[
p_r^{(0)}=2r,
\qquad
p_r^{(1)}=2r+1.
\] Set \(c_r=2r+1\), \(d_r=2r+2\) and \[
L_r=(2r-1,2r-3,\ldots,1,2,4,\ldots,2r),
\] \[
R_r=(2r+3,2r+5,\ldots,2k-1,2k,2k-2,\ldots,2r+4).
\] Then \[
W_{p_r^{(0)}}=c_rR_rd_rL_r,
\qquad
W_{p_r^{(1)}}=c_rL_rR_rd_r,
\] and \[
\overleftarrow{W}_{p_r^{(0)}}=
\operatorname{rev}(L_r)d_r\operatorname{rev}(R_r)c_r,
\] \[
\overleftarrow{W}_{p_r^{(1)}}=
d_r\operatorname{rev}(R_r)\operatorname{rev}(L_r)c_r.
\] At \(r=k-1\), \(R_r\) is empty and the formulas remain literal. The
block coordinate records the common junction index: \[
B_R(1)=0,
\qquad
B_R(p_r^{(0)})=B_R(p_r^{(1)})=r.
\]

\subsection{Necessity for the block criterion: unequal
blocks}\label{e.5.-necessity-for-the-block-criterion-unequal-blocks}

For three distinct horizontal labels \(x,y,z\), suppose the opener order
contains \[
y\prec_W z\prec_W x
\] while the closer stream requires \[
x\prec_{\leftarrow W} y\prec_{\leftarrow W} z.
\] No LIFO schedule can realize both. When \(x\) closes, both \(y\) and
\(z\) have already opened; after \(x\) is removed, \(z\) lies above
\(y\), yet the prescribed closer stream asks for \(y\) before \(z\).
Additional pushes cannot repair the order and \(z\) is not permitted to
close first.

Every unequal ordered block pair contains such a witness. Since
\(\overleftarrow{W}_b=\operatorname{rev}(W_b)\), it is enough to verify
\[
y\prec_{W_a}z\prec_{W_a}x,
\qquad
z\prec_{W_b}y\prec_{W_b}x.
\] The following table covers all ordered unequal-block cases directly.

\begin{longtable}[]{@{}
  >{\raggedright\arraybackslash}p{(\columnwidth - 8\tabcolsep) * \real{0.2000}}
  >{\raggedright\arraybackslash}p{(\columnwidth - 8\tabcolsep) * \real{0.2000}}
  >{\raggedright\arraybackslash}p{(\columnwidth - 8\tabcolsep) * \real{0.2000}}
  >{\raggedright\arraybackslash}p{(\columnwidth - 8\tabcolsep) * \real{0.2000}}
  >{\raggedright\arraybackslash}p{(\columnwidth - 8\tabcolsep) * \real{0.2000}}@{}}
\toprule\noalign{}
\begin{minipage}[b]{\linewidth}\raggedright
ordered case
\end{minipage} & \begin{minipage}[b]{\linewidth}\raggedright
range
\end{minipage} & \begin{minipage}[b]{\linewidth}\raggedright
\((x,y,z)\)
\end{minipage} & \begin{minipage}[b]{\linewidth}\raggedright
order in \(W_a\)
\end{minipage} & \begin{minipage}[b]{\linewidth}\raggedright
order in \(W_b\)
\end{minipage} \\
\midrule\noalign{}
\endhead
\bottomrule\noalign{}
\endlastfoot
\(1/p_s^{(0)}\) & \(1\le s\le k-1\) & \((2,1,2s+1)\) & \(1<2s+1<2\) &
\(2s+1<1<2\) \\
\(1/p_s^{(1)}\) & \(1\le s\le k-1\) & \((2,1,2s+1)\) & \(1<2s+1<2\) &
\(2s+1<1<2\) \\
\(p_r^{(0)}/1\) & \(1\le r\le k-1\) & \((2,2r+1,1)\) & \(2r+1<1<2\) &
\(1<2r+1<2\) \\
\(p_r^{(1)}/1\) & \(1\le r\le k-1\) & \((2,2r+1,1)\) & \(2r+1<1<2\) &
\(1<2r+1<2\) \\
\(p_r^{(0)}/p_s^{(0)}\) or \(p_r^{(0)}/p_s^{(1)}\) & \(r<s\) &
\((1,2r+1,2r+3)\) & \(2r+1<2r+3<1\) & \(2r+3<2r+1<1\) \\
\(p_r^{(1)}/p_s^{(0)}\) or \(p_r^{(1)}/p_s^{(1)}\) & \(r<s\) &
\((2s,1,2r+3)\) & \(1<2r+3<2s\) & \(2r+3<1<2s\) \\
\(p_r^{(0)}/p_s^{(0)}\) or \(p_r^{(1)}/p_s^{(0)}\) & \(r>s\) &
\((1,2r+1,2r-1)\) & \(2r+1<2r-1<1\) & \(2r-1<2r+1<1\) \\
\(p_r^{(0)}/p_s^{(1)}\) or \(p_r^{(1)}/p_s^{(1)}\) & \(r>s\) &
\((2r,2r+1,1)\) & \(2r+1<1<2r\) & \(1<2r+1<2r\) \\
\end{longtable}

All witnesses lie in range. If \(r<s\), then \(r\le k-2\) and
\(2r+3\le2k-1\); if \(r>s\), then \(r\ge2\) and \(2r-1\ge3\);
singleton/non-singleton cases require only \(k\ge2\). Hence every
unequal-block pair fails H compatibility.

\subsection{Sufficiency for the block criterion: same-block
schedules}\label{e.6.-sufficiency-for-the-block-criterion-same-block-schedules}

For the singleton pair \((1,1)\), push all openers in \(W_1\) and pop
them in reverse order.

Within a nonzero block \(r\), the diagonal pairs \[
(p_r^{(0)},p_r^{(0)}),
\qquad
(p_r^{(1)},p_r^{(1)})
\] again use push-all/pop-all. For \((p_r^{(0)},p_r^{(1)})\), use \[
\text{push }c_r,R_r,d_r;\quad
\text{pop }d_r,\operatorname{rev}(R_r);\quad
\text{push }L_r;\quad
\text{pop }\operatorname{rev}(L_r);\quad
\text{pop }c_r.
\] For \((p_r^{(1)},p_r^{(0)})\), use \[
\text{push }c_r,L_r;\quad
\text{pop }\operatorname{rev}(L_r);\quad
\text{push }R_r,d_r;\quad
\text{pop }d_r,\operatorname{rev}(R_r);\quad
\text{pop }c_r.
\] Every requested closer is the stack top, every label is used once,
and the stack ends empty. Thus equal block coordinate implies H
compatibility.

Combining E.5 and E.6 gives \[
\operatorname{HCompat}(k;a,b)
\iff
B_R(a)=B_R(b).
\] Together with E.3 this proves Theorem 7.2. No finite compatibility
graph or count formula is used as a premise.

\begin{center}\rule{0.5\linewidth}{0.5pt}\end{center}

\section{Detailed Boundary Proof and the Eight Physical Singleton States}\label{appendix-f.-detailed-boundary-proof-and-the-eight-physical-singleton-states}

This appendix uses only the \(n=2\) physical map.

\subsection{Physical projection and crease
census}\label{f.1.-physical-projection-and-crease-census}

The crease-incidence cycle is \[
A_1-H_1-B_1-L_1-B_2-H_2-A_2-U_1-A_1.
\] The four physical labels are \(h_0,h_1,u_0,l_0\). Hence the assigned
map depends only on \(\pi_1(P)=(h_0,h_1,u_0,l_0)\), and each physical
assignment has four formal representatives obtained by arbitrary choices
of the two absent formal bits.

The complete global test consists of the four primitive M/V precedences
and exactly two Butterfly pairs: \((H_1,H_2)\) and \((U_1,L_1)\). These
are the two pairs of opposite edges of the cycle.

\subsection{Negative product: Hamiltonian uniqueness and
validity}\label{f.2.-negative-product-hamiltonian-uniqueness-and-validity}

Traverse the cycle as \(A_1\to B_1\to B_2\to A_2\to A_1\) and assign
sign \(+1\) to a primitive edge pointing with this traversal and \(-1\)
to an edge pointing against it. Direct checkerboard evaluation gives
traversal signs \[
h_0,\quad -l_0,\quad h_1,\quad -u_0,
\] whose product is \(h_0h_1u_0l_0=\delta_0\).

If \(\delta_0=-1\), exactly one or three edges oppose the traversal. In
either orientation, the three majority edges form a directed Hamiltonian
path \(v_0\prec v_1\prec v_2\prec v_3\) and the remaining edge is the
shortcut \(v_0\prec v_3\). Thus there is exactly one linear extension.
The two opposite edge pairs have endpoint intervals \([0,1]\) with
\([2,3]\) and \([1,2]\) with \([0,3]\), respectively; one pair is
disjoint and the other nested. Hence the unique order passes both
Butterfly tests.

\subsection{Positive product: direct global
obstruction}\label{f.3.-positive-product-direct-global-obstruction}

If \(\delta_0=+1\), the number of traversal reversals is \(0,2,\) or
\(4\). With \(0\) or \(4\) reversals the primitive precedences form a
directed \(4\)-cycle. With exactly two reversals, the orientation is
acyclic but every linear extension makes one opposite-edge pair
alternate.

If the two reversed edges are adjacent, the poset has one source, one
sink, and two incomparable middle vertices. The two possible middle
orders force intervals \([0,2]\) and \([1,3]\) for one of the
opposite-edge pairs. If the reversed edges are opposite, the poset has
two sources followed by two sinks; whichever internal source/sink order
is chosen, one opposite-edge pair again has intervals \([0,2]\) and
\([1,3]\). Hence no globally valid state exists.

\subsection{Exact physical state
table}\label{f.4.-exact-physical-state-table}

\begin{longtable}[]{@{}ll@{}}
\toprule\noalign{}
physical projection \((h_0,h_1,u_0,l_0)\) & exact state set \\
\midrule\noalign{}
\endhead
\bottomrule\noalign{}
\endlastfoot
\(MMMV\) & \(\{(A_1,B_1,B_2,A_2)\}\) \\
\(MMVM\) & \(\{(B_2,A_2,A_1,B_1)\}\) \\
\(MVMM\) & \(\{(A_1,A_2,B_2,B_1)\}\) \\
\(MVVV\) & \(\{(A_2,A_1,B_1,B_2)\}\) \\
\(VMMM\) & \(\{(B_2,B_1,A_1,A_2)\}\) \\
\(VMVV\) & \(\{(B_1,B_2,A_2,A_1)\}\) \\
\(VVMV\) & \(\{(B_1,A_1,A_2,B_2)\}\) \\
\(VVVM\) & \(\{(A_2,B_2,B_1,A_1)\}\) \\
\end{longtable}

The other eight physical projections have empty exact state set. The
table is a literal presentation of the analytic Hamiltonian rule, not a
computational premise.


\begin{thebibliography}{99}

\bibitem{Nishat2013}
R.~I. Nishat.
\newblock \emph{Map Folding}.
\newblock M.Sc. thesis, University of Victoria, 2013.

\bibitem{NishatWhitesides2013}
R.~I. Nishat and S. Whitesides.
\newblock Map folding.
\newblock In \emph{Proceedings of the 25th Canadian Conference on Computational Geometry (CCCG 2013)}, pages 49--54, 2013.

\bibitem{Morgan2012}
T.~D. Morgan.
\newblock \emph{Map Folding}.
\newblock M.Eng. thesis, Massachusetts Institute of Technology, 2012.

\bibitem{JiaMitani2026}
Y. Jia and J. Mitani.
\newblock Algebraic characterization of $2\times n$ map foldability via order extensions of $1\times 2n$ strips.
\newblock \emph{JP Journal of Algebra, Number Theory and Applications}, 65(3):419--439, 2026.
\newblock doi:10.17654/0972555526022.

\bibitem{JiaMitaniUehara2020}
Y. Jia, J. Mitani, and R. Uehara.
\newblock Efficient algorithm for $2\times n$ map folding with a box-pleated crease pattern.
\newblock \emph{Journal of Information Processing}, 28:806--815, 2020.
\newblock doi:10.2197/ipsjjip.28.806.

\bibitem{AkitayaDemaineKu2026}
H.~A. Akitaya, E.~D. Demaine, and J.~S. Ku.
\newblock Computing flat-folded states.
\newblock In \emph{Origami8, Volume III: Proceedings of the 8th International Meeting on Origami in Science, Mathematics and Education}, Lecture Notes in Mechanical Engineering, pages 201--219. Springer Singapore, 2026.
\newblock doi:10.1007/978-981-96-6561-7\_14.

\bibitem{AsanoEtAl2010}
T. Asano, E.~D. Demaine, M.~L. Demaine, and R. Uehara.
\newblock Kaboozle is NP-complete, even in a strip.
\newblock In \emph{FUN 2010}, Lecture Notes in Computer Science 6099, pages 28--36. Springer, 2010.
\newblock doi:10.1007/978-3-642-13122-6\_5.

\bibitem{HullEtAl2025}
T.~C. Hull, A. Ibrahim, J. Paltrowitz, N. Ter-Saakov, and G. Wang.
\newblock The stamp folding problem from a mountain-valley perspective.
\newblock \emph{Discrete Mathematics \& Theoretical Computer Science}, 27(3), Article 16, 2025.
\newblock doi:10.46298/dmtcs.15454.

\bibitem{HoshidoEtAl2025}
J. Hoshido, T. Kamata, T. Ansai, and R. Uehara.
\newblock Computational complexity of one-dimensional origami with constraints on thickness at creases.
\newblock \emph{IEICE Transactions on Fundamentals of Electronics, Communications and Computer Sciences}, E108-A(9):1084--1091, 2025.
\newblock doi:10.1587/transfun.2024DMP0004.

\bibitem{Uehara2011}
R. Uehara.
\newblock Stamp foldings with a given mountain-valley assignment.
\newblock In \emph{Origami 5: Fifth International Meeting of Origami Science, Mathematics, and Education}, pages 585--597. CRC Press, 2011.

\bibitem{UmesatoEtAl2013}
T. Umesato, T. Saitoh, R. Uehara, H. Ito, and Y. Okamoto.
\newblock The complexity of the stamp folding problem.
\newblock \emph{Theoretical Computer Science}, 497:13--19, 2013.
\newblock doi:10.1016/j.tcs.2012.08.006.

\bibitem{Ulfarsson2012}
H.~{\'U}lfarsson.
\newblock Describing West-3-stack-sortable permutations with permutation patterns.
\newblock \emph{S{\'e}minaire Lotharingien de Combinatoire}, 67:Article B67d, 2012.

\bibitem{KTT2026}
J.~S. Ku, A. Terao, and K.~N. Terao.
\newblock An algebraic approach to layer ordering constraints for origami flat-foldability.
\newblock In \emph{Origami8, Volume III: Proceedings of the 8th International Meeting on Origami in Science, Mathematics and Education}, Lecture Notes in Mechanical Engineering, pages 317--333. Springer Singapore, 2026.
\newblock doi:10.1007/978-981-96-6561-7\_21.

\end{thebibliography}
\end{document}